%% file: paper1.tex
\documentclass[draft]{siamltex1213}
\usepackage{amsmath}
\usepackage{amssymb}
\usepackage{amsfonts}
\usepackage{enumerate}
\usepackage{epstopdf}
\usepackage[noend]{algorithmic}
\usepackage{algorithm,caption}

\usepackage{array}
\usepackage{subfig}
\usepackage[usenames,dvipsnames,svgnames]{xcolor}
\usepackage{tikz}
\usepackage{pgfplots}
\usepackage{multirow}
\usepackage{lipsum}
\usepackage[margin=3cm]{geometry}

\title{Fast algorithms for the computation of Fourier extensions of arbitrary length}

\author{Roel Matthysen, Daan Huybrechs}

\newcommand{\CM}{\bar{\DM}}

\newcommand{\DM}{A}

\newcommand{\e}{e}
\renewcommand{\imath}{\mathrm{i}}

\newcommand{\frameb}{\mathcal{G}}

\newcommand{\dlambda}{\upsilon}
\newcommand{\dLambda}{\Upsilon}
\newcommand{\jn}{k}
\newcommand{\jm}{l}
\newcommand{\sn}{n}
\newcommand{\bN}{N}
\newcommand{\sm}{m}
\newcommand{\bM}{M}
\newcommand{\TRID}{Exp}
\newcommand{\PA}{Imp}

\newcommand{\cutoff}{\tau}
\newcommand{\resnorm}{$||\mathcal{F}(f)-f||$}
\newcommand{\tnlogregion}{$\sigma_i,\sigma_{NM/L},\sigma_j $}
\newtheorem{thm}{Theorem}

\newtheorem{remark}{Remark}

\newlength{\figurewidth}
\newlength{\figureheight}

\date{\today} 

\begin{document}

\maketitle

\begin{abstract}

Fourier series of smooth, non-periodic functions on $[-1,1]$ are known to
exhibit the Gibbs phenomenon, and exhibit overall slow convergence.
One way of overcoming these problems is by using a Fourier series on a larger
domain, say $[-T,T]$ with $T>1$, a technique called Fourier extension or Fourier
continuation. When constructed as the discrete least squares minimizer in
equidistant points, the Fourier extension has been shown shown to converge
geometrically in the truncation parameter $N$. A fast ${\mathcal O}(N \log^2 N)$
algorithm has been described to compute Fourier extensions for the case where
$T=2$, compared to ${\mathcal O}(N^3)$ for solving the dense discrete least
squares problem. We present two ${\mathcal O}(N\log^2 N )$ algorithms for the
computation of these approximations for the case of general $T$, made possible
by exploiting the connection between Fourier extensions and Prolate Spheroidal
Wave theory. The first algorithm is based on the explicit computation of
so-called periodic discrete prolate spheroidal sequences, while the second
algorithm is purely algebraic and only implicitly based on the theory. 
\end{abstract}

\section{Introduction}

Fourier series are a good choice for the approximation of a smooth periodic function on a bounded interval. They offer exponential convergence, good frequency resolution, and the approximation can be computed numerically via the FFT. However, when the function is smooth but non-periodic, the exponential convergence of a Fourier series over the interval is lost, and ringing artefacts known as the Gibbs phenomenon are introduced.

The Fourier extension technique (FE) \cite{Boyd2002, Boyd2005, Bruno2003, Bruno2007} aims to transfer the desirable properties
of Fourier series for periodic functions to the non-periodic case.  The 
principle is to approximate a non-periodic function that is defined on $[-1,1]$
by a Fourier series that is periodic on $[-T,T]$. While the approximation may
vary wildly in $[-T,-1[$ and $]1,T]$, under certain conditions it is guaranteed
to converge exponentially to the original function within the interval. An
illustration is shown in Figure \ref{fig:fexample}, where the extension is seen
to agree closely with the given function on $[-1,1]$. Outside this interval, the
extension is arbitrary, and in most cases defined by the solution method.

The main difficulty with this technique is the ill-conditioning of the
restricted  Fourier basis. Numerically, this leads to ill-conditioned linear
systems, which are difficult to solve efficiently.

\setlength{\figurewidth}{8cm}
\setlength{\figureheight}{4cm}
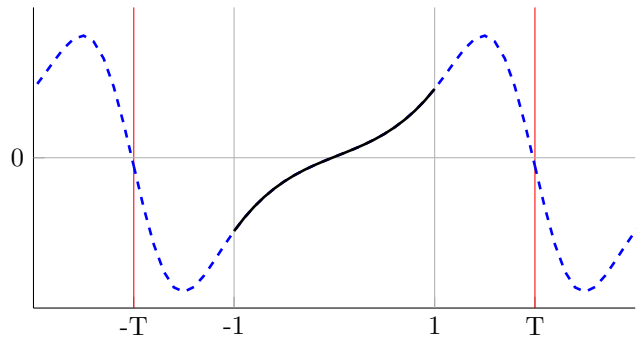
\begin{figure}[htbp]
\begin{center}
\input{img/figure3}
\end{center}
\caption{A periodic extension to $[-T,T]$ of a smooth function on $[-1,1]$.}
\label{fig:fexample}
\end{figure}

These constructions have been known in embedded or fictitious domain methods
 for general bases. A study on the approximation properties of
extensions in the Fourier basis by Boyd \cite{Boyd2002} revealed that inside the
smaller interval, the
extension  can be exponentially converging to the function. Further, he proposed
the truncated singular value decomposition as a robust method
for computing extensions from equispaced data. The resulting scheme, named
FPIC-SU, can compute extensions of functions in the smaller interval, that are
exact almost up to machine precision. The convergence rate is only limited by
the smoothness of the function. Bruno et.\ al.\ used the same principle when
approximating surfaces by Fourier series on extended domains\cite{Bruno2003}. This provided a starting point for the very efficient FC-Gram method
\cite{Bruno2010, Lyon2010, albin2011spectral}.

Exponential error convergence of the FE problem was proven in
\cite{Huybrechs2010} when inverting the Grammian matrix of the continuous least squares problem. Later, Adcock et. al extended the convergence analysis to the discrete least squares problem on equispaced data. At least superalgebraic convergence was proven for
analytic functions, when using the truncated SVD.

From an implementation point of view, the cost of the full SVD required for the FPIC-SU is prohibitively large. Recently, an FE algorithm was introduced by Lyon \cite{Lyon2011} that computes extensions in ${\mathcal O}(N\log^2{N})$ time. However, this algorithm only produces extensions of double length, i.\ e.\ it is unique to the case $T=2$. Due to the reliance on symmetries only present when $T$ is a power of 2, the algorithm cannot easily be extended to arbitrary $T$. 

In this paper we present fast algorithms for the computation of Fourier extensions of arbitrary extension length. 
One argument for varying the parameter $T$ is found in the resolution power
of the extension. The number of degrees of freedom per wavelength to represent
an oscillatory function approaches the optimal value of 2 as $T$ approaches $1$
from above \cite{Adcock2011}. When $T=2$, that number has already
doubled. Other arguments may be performance related, as one may for example tune
the length of the FFTs that are used in the computations, or a restriction on the data may be found in the application itself.

Our algorithms stem from connecting the FE problem with classical results from
signal processing theory. We state how it is essentially equivalent to the
problem of bandlimited extrapolation. Central in this discussion are
the so-called Prolate Spheroidal Wave functions, originally introduced at Bell labs in the
1970s \cite{Slepian1978a, Landau1961, Pollack1961}. The study of these
functions and their special properties has been an active domain in signal processing since.

The connection with Prolate Spheroidal Wave theory leads to explicit
formulations for eigenvectors of the FE problem. Analysis of the FE problem
learns that the relevant ill-conditioning can be captured using just
${\mathcal O}(\log{N})$ of these vectors. Combined with a fast solver for the remaining
well-conditioned problem this leads to ${\mathcal O}(N\log^2{N})$ solvers.

We explore two approaches in detail. In the first approach, a set of ${\mathcal O}(\log N)$ \emph{discrete prolate spheroidal sequences} is explicitly computed. This is based on their known properties: the vectors are known to be eigenvectors of a tridiagonal matrix. The second approach is purely algebraic and is not explicitly based on prolate spheroidal wave theory. Hence, this approach is more generally applicable. Both approaches have comparable performance, with a slight edge for the first approach.

\subsection{Overview of the paper}

 We formally state the FE problem in \S\ref{sec:fourierext} and summarize relevant previous results. A concise overview of Prolate Spheroidal Wave functions is presented in \S\ref{sec:slepian}, including subsequent work on discrete variants. The
 connection between these discrete variants and the FE problem is used to
 obtain fast algorithms in \S\ref{sec:algorithms}. Finally, we present numerical experiments to illustrate the numerical performance of these algorithms in \S\ref{sec:results}.

\input{fourierext}
\input{dpss}
\input{algorithms}

\input{results}

\bibliography{mrabbrev,library}
\bibliographystyle{abbrv}
\end{document}

%% file: img/figure3.tex
%
%
\begin{tikzpicture}

\begin{axis}[%
width=\figurewidth,
height=\figureheight,
scale only axis,
xmin=-3,
xmax=3,
xtick={-2,-1,1,2},
xticklabels={-T,-1,1,T},
ymin=-1.2,
ymax=1.2,
ytick={0},
axis x line*=bottom,
axis y line*=left
]
\addplot [
color=red,
solid,
forget plot
]
table[row sep=crcr]{
-2 -1.2\\
-2 1.2\\
};
\addplot [
color=red,
solid,
forget plot
]
table[row sep=crcr]{
2 -1.2\\
2 1.2\\
};
\addplot [
color=white!70!black,
solid,
forget plot
]
table[row sep=crcr]{
-1 -1.2\\
-1 1.2\\
};
\addplot [
color=white!70!black,
solid,
forget plot
]
table[row sep=crcr]{
1 -1.2\\
1 1.2\\
};
\addplot [
color=white!70!black,
solid,
forget plot
]
table[row sep=crcr]{
-3 0\\
3 0\\
};
\addplot [
color=blue,
dashed,
line width=1.0pt,
forget plot
]
table[row sep=crcr]{
-6 -0.0670435368132578\\
-5.9 -0.40194964697431\\
-5.8 -0.693593669396541\\
-5.7 -0.910058969741345\\
-5.6 -1.03527004685047\\
-5.5 -1.07057293565218\\
-5.4 -1.03154365051825\\
-5.3 -0.941656592429248\\
-5.2 -0.825292008123391\\
-5.1 -0.70235396818632\\
-5 -0.585735023107513\\
-4.9 -0.481560920197492\\
-4.8 -0.391172589679631\\
-4.7 -0.313498444863129\\
-4.6 -0.246809682987811\\
-4.5 -0.18951150004584\\
-4.4 -0.140201963408321\\
-4.3 -0.0974684039247736\\
-4.2 -0.0597783348777481\\
-4.1 -0.0255462043729277\\
-4 0.00675006696498824\\
-3.9 0.0385462327629888\\
-3.8 0.0712782098587375\\
-3.7 0.106468325289677\\
-3.6 0.145702070910644\\
-3.5 0.190511613724997\\
-3.4 0.242309572173317\\
-3.3 0.302498322036329\\
-3.2 0.372672797897078\\
-3.1 0.454560812542801\\
-3 0.54923504330363\\
-2.9 0.655371207672469\\
-2.8 0.76690482336103\\
-2.7 0.871093623144101\\
-2.6 0.948323307025697\\
-2.5 0.974681402618259\\
-2.4 0.927344299482684\\
-2.3 0.79152503116718\\
-2.2 0.566708142267839\\
-2.1 0.269708842911472\\
-2 -0.0670435368132572\\
};
\addplot [
color=blue,
dashed,
line width=1.0pt,
forget plot
]
table[row sep=crcr]{
-2 -0.0670435368132578\\
-1.9 -0.40194964697431\\
-1.8 -0.693593669396541\\
-1.7 -0.910058969741345\\
-1.6 -1.03527004685047\\
-1.5 -1.07057293565218\\
-1.4 -1.03154365051825\\
-1.3 -0.941656592429248\\
-1.2 -0.825292008123391\\
-1.1 -0.70235396818632\\
-1 -0.585735023107513\\
-0.9 -0.481560920197492\\
-0.8 -0.391172589679631\\
-0.7 -0.313498444863129\\
-0.6 -0.246809682987811\\
-0.5 -0.18951150004584\\
-0.4 -0.140201963408321\\
-0.3 -0.0974684039247736\\
-0.2 -0.0597783348777481\\
-0.1 -0.0255462043729277\\
0 0.00675006696498824\\
0.1 0.0385462327629888\\
0.2 0.0712782098587375\\
0.3 0.106468325289677\\
0.4 0.145702070910644\\
0.5 0.190511613724997\\
0.6 0.242309572173317\\
0.7 0.302498322036329\\
0.8 0.372672797897078\\
0.9 0.454560812542801\\
1 0.54923504330363\\
1.1 0.655371207672469\\
1.2 0.76690482336103\\
1.3 0.871093623144101\\
1.4 0.948323307025697\\
1.5 0.974681402618259\\
1.6 0.927344299482684\\
1.7 0.79152503116718\\
1.8 0.566708142267839\\
1.9 0.269708842911472\\
2 -0.0670435368132572\\
};
\addplot [
color=blue,
dashed,
line width=1.0pt,
forget plot
]
table[row sep=crcr]{
2 -0.0670435368132578\\
2.1 -0.40194964697431\\
2.2 -0.693593669396541\\
2.3 -0.910058969741345\\
2.4 -1.03527004685047\\
2.5 -1.07057293565218\\
2.6 -1.03154365051825\\
2.7 -0.941656592429248\\
2.8 -0.825292008123391\\
2.9 -0.70235396818632\\
3 -0.585735023107513\\
3.1 -0.481560920197492\\
3.2 -0.391172589679631\\
3.3 -0.313498444863129\\
3.4 -0.246809682987811\\
3.5 -0.18951150004584\\
3.6 -0.140201963408321\\
3.7 -0.0974684039247736\\
3.8 -0.0597783348777481\\
3.9 -0.0255462043729277\\
4 0.00675006696498824\\
4.1 0.0385462327629888\\
4.2 0.0712782098587375\\
4.3 0.106468325289677\\
4.4 0.145702070910644\\
4.5 0.190511613724997\\
4.6 0.242309572173317\\
4.7 0.302498322036329\\
4.8 0.372672797897078\\
4.9 0.454560812542801\\
5 0.54923504330363\\
5.1 0.655371207672469\\
5.2 0.76690482336103\\
5.3 0.871093623144101\\
5.4 0.948323307025697\\
5.5 0.974681402618259\\
5.6 0.927344299482684\\
5.7 0.79152503116718\\
5.8 0.566708142267839\\
5.9 0.269708842911472\\
6 -0.0670435368132572\\
};
\addplot [
color=black,
solid,
line width=1.0pt,
forget plot
]
table[row sep=crcr]{
-1 -0.58575\\
-0.9 -0.4815\\
-0.8 -0.39125\\
-0.7 -0.3135\\
-0.6 -0.24675\\
-0.5 -0.1895\\
-0.4 -0.14025\\
-0.3 -0.0975\\
-0.2 -0.05975\\
-0.1 -0.0255\\
0 0.00675\\
0.1 0.0385\\
0.2 0.07125\\
0.3 0.1065\\
0.4 0.14575\\
0.5 0.1905\\
0.6 0.24225\\
0.7 0.3025\\
0.8 0.37275\\
0.9 0.4545\\
1 0.54925\\
};
\end{axis}
\end{tikzpicture}%

%% file: fourierext.tex
\section{Fourier extensions}
\label{sec:fourierext}

\subsection{Problem formulation}

For the remainder of this paper we will focus on infinitely differentiable non-periodic functions $f$ on the interval $[-1,1]$. 
Other intervals are easily dealt with through affine transformations. 

The approximation is constructed on the extended interval $[-T,T]$, with $T$ the extension parameter. 
For ease of notation denote 
 \begin{equation*}
\phi_k(x)=\frac{1}{\sqrt{2T}}\e^{\imath\frac{\jn\pi}{T}x}.
\end{equation*}
Then \begin{equation*}
\frameb_N =\left\{\phi_k\right\}_{\jn=-\sn,\dots,\sn}
\end{equation*}
 is the set of $\bN=2\sn+1$ Fourier basis functions on the extended interval. 
The Fourier extension problem is now formalised as finding an approximation $F_\bN(f)$ such that
\begin{equation}
F_\bN(f)=\min_{g\in\frameb_N}||f-g||_{[-1,1]}.
\label{eqn:febase}
\end{equation}

We note that the infinite set $\frameb_\infty$ is a basis for $\mathcal{L}^2$ on the extended interval $[-T,T]$, but it constitutes a \emph{frame} on $[-1,1]$ \cite{Duffin1952}, i.e. it is redundant. This reflects the fact that a function can be extended in different ways. Once truncated to finite $N$, $\frameb_N$ is actually a basis on $[-1,1]$, albeit a very ill-conditioned one. We refer to \cite{Adcock2012} for a detailed discussion of this point of view.

In an implementation context, it is more natural to consider the search for the Fourier coefficients $a\in\mathbb{C}^N$, 
\begin{equation*}
a=\min_{c\in\mathbb{C}^N}||f-\sum_{\jn=-n}^nc_k\phi_k||_{[-1,1]}.
\end{equation*}
Depending on the norm to be minimised, the extension problem takes on a different formulation:

\paragraph{Discrete Fourier extensions}
Most practical applications provide information about $f$ as samples in a
predefined set of points. 
The norm in (\ref{eqn:febase}) is then conveniently replaced by a discrete least squares norm
\begin{equation}
\tilde{F}_\bN(f)=\min_{g\in\frameb_N}\sum_\jm\left(f(x_\jm)-g(x_\jm)\right)^2,\qquad x_\jm\in[-1,1].
\label{eqn:febase2}
\end{equation}
Previous work distinguishes between uniform sampling, and an optimal sampling
set called \emph{mapped symmetric Chebychev nodes} \cite{Adcock2011}. 
The focus of this paper is on the case where the sampling points are uniform,
\begin{equation*}
x_\jm=\frac{\jm}{\sm},\qquad \jm=-\sm,\dots,\sm,
\end{equation*}
for a total of $M=2m+1$ points. Typically, some oversampling $\gamma$ is used so
$M=\gamma N > N$. 
Including appropriate normalisation, this leads to the matrix formulation that is used throughout the rest of the paper. Let
\begin{equation}
\DM_{\jm,\jn}=\phi_\jn(x_\jm)=\frac{1}{\sqrt{2T\sm}}\e^{\imath\frac{\pi \jn \jm}{T\sm}},\qquad b_\jm=\frac{1}{\sqrt{\sm}}f(x_\jm).
\label{eqn:discretefe}
\end{equation}
Note that matrix-vector products involving $\DM$ can be performed very efficiently using FFTs of length $L=2T\sm$. Therefore $L$ is assumed to be integer for the remainder of this paper.

The Fourier coefficients are then found by collocation, through solving the rectangular system
\begin{equation}\label{eq:system_discrete}
\DM a\approx b
\end{equation}
in a least squares sense.
Note that due to the nature of the problem we are not interested in a
specific solution $a$, just one of the possibly many solutions that lead to a
small residual $||Aa-b||$. Solving this problem efficiently is the focal point of
this paper.

\paragraph{Continuous Fourier extensions}
When the norm is the regular $\mathcal{L}^2$-norm over $[-1,1]$, 
\begin{equation*}
\overline{F}_N(f)=\min_{g\in\frameb_N}||f-g||_{2,[-1,1]}
\end{equation*} is known as the \emph{continuous Fourier extension}. The resulting least-squares problem can be solved by formulating the Grammian matrix. Let 
\begin{equation}
\label{eqn:continuousfe}
\CM_{i,j}=\int_{-1}^1\phi_i(x)\overline{\phi_j(x)}dx=\frac{\sin{\frac{\pi(i-j)}{T}}}{\pi(i-j)},\qquad b_i=\int_{-1}^1f\overline{\phi_i(x)}dx.
\end{equation}
Then the Fourier coefficients follow from
\begin{equation}\label{eq:system_continuous}
\CM a=b.
\end{equation}

\subsection{Convergence and stability}

The usefulness of approximation schemes hinges on two properties: the speed of convergence to the given function, and the stability of the required computations. 
Considerable effort has been put into quantifying these properties both analytically and numerically \cite{Huybrechs2010}  \cite{Adcock2011} \cite{Adcock2012} \cite{Lyon2012a}. 
Without going into too much detail, we recap the most important results.

First of all, a distinction should be made between the exact discrete and continuous FE solutions $\tilde{F}_N(f)$ and $\overline{F}_N(f)$, and their computer-implemented counterparts. 
Computing the exact solutions is known to be unstable, as they can grow
unbounded outside the interval of interest. Numerical algorithms however will never compute these exact solutions. 
Due to regularisation, the numerical FEs $\tilde{G}_N(f)$ and
$\overline{G}_N(f)$ are more stable, while maintaining the desired
convergence behaviour.
In \cite{Adcock2012}, the truncated SVD was used as a model scheme for solving
the ill-conditioned systems (\ref{eqn:discretefe}) and
(\ref{eqn:continuousfe}). Given the Singular Value Decomposition of the FE matrix
$\DM=USV'$, the solution to $\DM a=b$ is
\begin{equation}
  \label{eq:1}
a=VS^\dagger U'b,\qquad
S^\dagger_{i,i}=\begin{cases}
  \frac{1}{S_{i,i}} & S_{i,i}>\cutoff \\
0 & \mbox{otherwise}.
\end{cases}
\end{equation}
The truncation parameter $\cutoff$ then acts as a regularisation parameter. It
is usually chosen close to the machine precision. 

\paragraph{Stability}
Following \cite{Adcock2012}, stability is defined in terms of the absolute condition number of the FE mapping 
\begin{equation*}
\kappa(F_N)=\sup\{||F_N(b)||:b\in\mathbb{C}^N,||b||=1\},
\end{equation*}
where, with slight abuse of notation, $F_N(b)$ is the solution to the FE problem with right hand side $b$, and
$||\cdot||$ is the regular $l^2$ norm over $[-1,1]$. This condition number can
be computed for the continuous and discrete FEs, both the exact and numerical
versions. It was shown to grow exponentially in $N$ for the exact solution
to the continuous and discrete FE problem.


When looking at the numerical FEs the situation changes considerably. 
The condition number of the numerical continuous FE mapping is $\kappa(\overline{G}_N) \lesssim 1/\sqrt{\cutoff}$. 
 The condition number of the numerical discrete FE $\kappa(\tilde{G}_N)$ is dependent on a constant $0<a(\gamma;T)\leq1$, independent of $N$, that satisfies $a(\gamma;T)\to0$ as $\gamma\to\infty$ for fixed $T$. 
 It is given by $\kappa(\tilde{G}_N)\lesssim \cutoff^{-a(\gamma;T)}, \quad\forall N\in N$. 
 This means that for a sufficiently large oversampling factor $\gamma$, the condition number of the numerical FE mapping can be made reasonably close to 1. 

Meanwhile, the condition number of the matrices $\DM$ and $\CM$ grows exponentially as $N\to\infty$, where $M\geq N$ for the discrete FE. 
This is surprising, given the good condition of the FE mapping. 
It can be understood by noting that extensions with small coefficient norm and small residual are guaranteed to exist. 
The numerical algorithms will steer clear of the unstable exact solution, and instead return one of these alternatives. 
For a full exposition on the stability of FE calculations, see \cite{Adcock2012}.

\paragraph{Convergence}
Concerning convergence, results in \cite{Adcock2012} are valuable only for functions that are analytic in a region $\mathcal{D}(\rho^*)$ of the complex plane and continuous on its border. 
This region is a Bernstein ellipse under a transformation that allows Fourier extensions to be understood as polynomial approximations. 
For such functions $f$, the exact continuous and discrete FEs converge geometrically, with a speed \begin{equation*}
||f-\tilde{F}_N(f)|| \leq c_f\rho^{-N}.
\end{equation*}
Here $\rho=\min\{\rho^*,E(T)\}$ and $c_f$ is proportional to $\max_{x\in\mathcal{D}(\rho)}|f(x)|$. 
$E(T)$ is known as the Fourier extension constant and is given by $\cot^2\left(\frac{\pi}{4T}\right)$. 
Note that for the exact discrete FE, there is an added requirement of scaling
$M$ as $O(N^2)$, to avoid the Runge phenomenon.

Under the same analyticity conditions, the error decay of the numerical
counterparts $\overline{G}_N$ and $\tilde{G}_N$ can be broken down into several subregions:
\begin{enumerate}
\item If $N<N_2$, where $N_2$ is a function-independent breakpoint, $||f-\tilde{G}_N(f)||$ converges or diverges exponentially fast at the same rate as the exact solution.
\item When $N\leq N_0$ (continuous) or $N_2 \leq N \leq N_1:=2N_0$ (discrete), where $N_0$ is another function-independent breakpoint depending, both $||f-\overline{G}_N(f)||$ and $||f-\tilde{G}_N(f)||$ decay like $\rho^{-N}$.
\item When $N=N_0$ or $N=N_1$, the errors are approximately
\begin{equation*}
||f-\overline{G}_{N_0}(f)||\approx c_f(\sqrt{\cutoff})^{d_f}, ||f-\tilde{G}_{N_1}(f)||\approx c_f\cutoff^{d_f-a(\gamma;T)},
\end{equation*}
where $c_f$ is as before, and $d_f=\frac{\log{\rho}}{\log{E(T)}}\in (0,1]$.
\item When $N>N_0$ or $N>N_1$, the errors decay at least super algebraically fast down to \emph{maximal achievable accuracies} of order $\sqrt{\cutoff}$ and  $\cutoff^{1-a(\gamma;T)}$ respectively.
\end{enumerate}

This behaviour of the error offers insight into the usability of Fourier extensions. 
A first observation is that the continuous FE is limited to a maximal achievable accuracy of $\sqrt{\cutoff}$. 
Coupled with the need to compute Fourier integrals to compose the right hand side $b$ in (\ref{eqn:continuousfe}), this makes the continuous FE unfit for practical use. 
However, the algorithms presented in \S\ref{sec:algorithms} for the discrete FE can be adapted to this context with little extra effort. 
This is documented in \S\ref{sec:continuousalg}. 

On the contrary, the numerical discrete FE guarantees convergence up to a certain power of $\cutoff$. 
By varying this cutoff, the oversampling and the extension length $T$ this
maximum achievable accuracy can be made very close to the machine precision.

\subsection{Influence of the extension length}
\label{sec:extensionlength}
The main contributions of this paper are algorithms that add flexibility in the
choice of extension length $T$. The increased resolution power was already cited
as an argument to reduce $T$, but it is important to be aware of the possible consequences.
Therefore, in this section we summarize the influence of this parameter on
convergence, resolution power, and conditioning of the FE problem.

First note that the Fourier extension constant $E(T)$ grows with $T$. 
For functions analytic in a sufficiently large region, the convergence rate
$\rho$ is limited by this constant. Increasing $T$ thus increases the
convergence rate, and vice versa.

The resolution power of a scheme, first studied by Gottlieb and Orszag
\cite{Gottlieb1983} is a measure of the amount of point samples needed to resolve an oscillatory function to a certain precision. 
Let 
\begin{equation*}
\mathcal{R(\omega,\delta)}=\min\{N\in\mathbb{N}:||e^{i\pi\omega}-F_N(e^{i\pi\omega})||_{\infty}<\delta\},\quad \omega>0,
\end{equation*}
for some small $\delta$. 
Then $F_N$ has a resolution constant $r$ if
\begin{equation*}
\mathcal{R}(\omega,\delta)\sim r\omega, \quad \omega\to\infty.
\end{equation*}
For regular Fourier series, this constant has the optimal value 2.

A theoretical argument shows that for the continuous FE this resolution constant increases with $T$ \cite{Adcock2011}. 
More specifically, 
\begin{equation*}
r(T)\leq 2T\sin\left(\frac{\pi}{2T}\right), \quad T\in(1,\infty).
\end{equation*}
Thus, for $T\approx 1$ the resolution constant $r(t)\approx 2T$ is close to optimal. 
When $T$ tends to infinity, $r(T)\sim\pi$.  
It is even possible to optimally balance convergence speed with resolution power when aiming for a predetermined accuracy $\epsilon_{tol}$. 
This is achieved by varying $T$ with $N$, specifically
\begin{equation}
T(N,\epsilon_{tol})=\frac{\pi}{4}\left(\arctan(\epsilon_{tol})^{\frac{1}{2N}}\right)^{-1}.
\label{eqn:tvarieswn}
\end{equation}

Although no equivalent analysis exists for the discrete FE, there have been several attempts to determine FE parameters that are in some sense optimal.
In \cite{Bruno2007} Bruno suggests the values $T=2$ and $\gamma=2$ as a general rule of thumb, but at the same time notes that the optimal parameters are heavily function dependent. 
Note that increasing both the extension length $T$ and the oversampling $\gamma$ will likely increase the resolution constant $r$. 
Especially since it was shown in \cite{Adcock2011} that the limit $r(T)\sim\pi$ as $T\to\infty$ no longer holds for a discretised FE, instead the resolution constant grows as $r(T)\sim 2T$. 
Even though this was only observed for data points distributed as a variant of Chebychev points, it is indicative that the resolution constant for $T=2,\gamma=2$ will be considerably above the optimal value.

The precise interplay between $T$ and $\gamma$ on the one hand, and resolution power and conditioning on the other hand was studied in detail by Adcock and Rua \cite{Adcock2013}. 
They found that the condition number $\kappa(\tilde{G}_N)$ of the equispaced
discrete FE depends only on the product $T\gamma$. 
Increasing either will lower the condition number.
Thus as long as $\gamma$ is increased or decreased accordingly when varying $T$, the conditioning of the FE mapping remains constant. 
This is cited as an argument to limit $T$ to 2, to profit from the at the time only available fast FE algorithm. 

Furthermore, the resolution constant is also dependent on the product $T\gamma$, growing as $r(T)\sim T\gamma$.
This illustrates the tradeoff between resolution power and conditioning. 
Numerical experiments in \cite{Adcock2013} showed that by allowing the condition number to grow from $\kappa\approx10$ to $\kappa\approx100$, the resolution constant was halved, while further increasing $\kappa$ had very little additional value.
However, it should be noted that these experiments were only carried out for $T=2$.
Lifting the restriction on $T$ may thus offer more flexibility in finding a balance between resolution power and conditioning.

An interesting open problem raised in \cite{Adcock2013} is the possibility to vary $T$ with $M$, to achieve optimal resolution power in a manner similar to (\ref{eqn:tvarieswn}). 
Due to the lack of a fast algorithm, any gains from varying $T$ were considered of limited practical usability compared to the fast algorithm for $T=2$. 
The fast algorithms presented in this paper warrant a closer look at the possible benefits from this method.


%% file: dpss.tex
\section{Prolate spheroidal wave functions and discrete variants}
\label{sec:slepian}


A long-standing problem in signal processing theory is that of bandlimited
extrapolation. The problem is, assuming some portion of a bandlimited signal is
known, to accurately predict the missing data. In a first subsection
we explain how the FE problem is a specific variant of this problem. The subsequent
sections then explore the theory of Prolate Spheroidal Wave functions, that
plays a major role in bandlimited extrapolation, for further use in the FE algorithms.

\subsection{Discrete Fourier extensions and bandlimited extrapolation}
The discrete Fourier extension is closely related to discrete band limited
extrapolation, i.e., to the problem of reconstructing a discrete bandlimited
signal from a number of data samples. Simply put, we are looking for a vector $y$ with discrete Fourier
transform $Y$ such that
\begin{align*}
y[k]&\approx f[k] \qquad k\in S_{\bM,t} \\
Y[l]&=0 \qquad l\not\in S_{\bN,\omega},
\end{align*}
where $f$ is the sample data, and $S_{\bM,t}$ and $S_{\bN,\omega}$ are sampling
sets in the discrete time and frequency domains, of sizes $M$ and $N$
respectively. Depending on the problem parameters the solution might not be
unique. In this case an additional minimal solution norm constraint can be added \cite{jain1981}.

A popular method is the Papoulis-Gerschberg algorithm \cite{Gerchberg1974} \cite{Papoulis1975}. It uses a two-step iteration process to alternate matching the given data and complying with the frequency constraints. Variants that use the conjugate gradients and related methods to speed up the iteration process tend to perform reasonably well numerically \cite{Strohmer1995}. They operate at a cost of $O(N\log{N})$ operations per iteration, where the number of iterations scales with the bandwidth of the signal.

Besides these iterative methods, considerable attention was given to direct methods for solving this discrete problem. One such method is the one proposed by Jain and Ranganath \cite{jain1981}, where both the time and frequency sampling sets are contiguous, i.\ e.\ $S_t=-\sm,\dots,\sm$ and $S_\omega=-\sn,\dots,\sn$. They commence by writing the data as a function of the unknown coefficients:
\begin{align}
 f&=Jy=D_\bM B_\bN y \label{eqn:jaineqn}\\
D_{\bM,i,j}&=\delta_{i,j}, \quad |i|,|j|\leq \sm \label{eqn:ddef}\\
B_{\bN,p,l}&=\frac{1}{L}\sum_{\jn=-\sn}^\sn \exp{\imath\frac{(p-l)2\pi \jn}{L}}=\frac{1}{L}\frac{\sin\left(\frac{(p-l)N\pi}{L}\right)}{\sin\left(\frac{(p-l)\pi}{L}\right)}.\label{eqn:bdef}
\end{align}
Here, $B_\bN$ is a $L\times L$ circulant matrix that represents a discrete low-pass filter, and $D_\bM$ is a $\bM\times L$ selection operator. $\delta_{i,j}$ is the kronecker delta. $D_M'$ is an extension operator that pads a sequence of length $M$ with zeros to length $L$.

The direct methods of Jain and Ranganath then consist of solving $Jy=f$ in a least squares sense by formulating the normal equations
\begin{equation*}
J'Jy=J'f.
\end{equation*}

Since $J'J$ is symmetric positive definite and Toeplitz, the Levinson-Trench
algorithm can be applied to compute the inverse of $J'J$ in $O(L^2)$
operations. They also suggested another approach,  using the singular value decomposition  of $J$ to solve the least-squares problem. The resulting singular vectors were named periodic discrete prolate spheroidal sequences (P-DPSSs), after the prolate spheroidal wave functions (PSWFs) arising in continuous bandlimited extrapolation. 

A closer look at the matrix $\DM$ from the discrete Fourier extension (\ref{eqn:discretefe}) makes the relation with (\ref{eqn:jaineqn}) apparent. 
Adopting the notation for the DFT length $L=2T\sm$, 
\begin{align}
\notag(\DM\DM')_{pq}
&=\frac{1}{L}\sum_{\jn=-\sn}^\sn\exp{\imath\frac{(p-q)2\pi \jn}{L}}, \qquad p,q=-\sm,\dots,\sm\\
&=(D_\bM B_\bN D_\bM')_{pq}=(JJ')_{pq},
\label{eqn:aaisjj}
\end{align}
where the last line follows from the idempotency of $B_\bN$. Consequently, $A$
and $J$ share the same left singular vectors, and the same singular
values. Essentially, discrete Fourier extension is a reformulation of the
bandlimited extrapolation problem with the $N$ frequency coefficients as
unknowns, instead of the extrapolated signal. The focus has also shifted from
determining an accurate extension to the approximation of the given data samples.

 The next sections concern the PSWFs and P-DPSSs, and how they are natural solutions for the bandlimited extrapolation problem. The groundwork for this theory was  detailed in a series of papers by Slepian, Landau and Pollak from 1961 onwards \cite{Slepian1978a,Landau1961,Pollack1961,Pollack1961a,Slepian1978}.
An overview is given in \cite{Slepian1983}.
\subsection{Prolate Spheroidal Wave Functions}
\label{sec:pswf}
%

Denote by $f(x)$ and $\mathcal{F}(\xi)$ a function in $\mathcal{L}^2$ and its Fourier transform, so that
 \begin{equation*}
F(\xi)=\int_{-\infty}^\infty f(x)e^{-2\pi i x \xi}ds,\qquad f(x)=\int_{-\infty}^\infty F(\xi)e^{2\pi i x \xi}d\xi.
\end{equation*}
The time- and bandlimiting operators $\mathcal{D}$ and $\mathcal{B}$ are then defined as 
\begin{equation}
\label{eqn:timebandlimit}
\mathcal{D}f(x)=\hat{f}(x)=\begin{cases}f(x) & |x|\leq T \\ 0 & |x|>T \end{cases}\qquad \mathcal{B}f(x)=\int_{-\Omega}^{\Omega}F(\xi)e^{i 2\pi \xi x }d\xi,
\end{equation}
which project onto $\mathcal{L}^2_{[-T,T]}$ and $PW_{\Omega}$, the Paley-Wiener space of bandlimited functions, respectively. 
Note that the bandlimiting operator can also be written as
\begin{equation*}
\mathcal{B}f(x)=\int_{-\infty}^\infty f(s)\frac{\sin(2\pi\Omega(x-s))}{\pi(x-s)}ds.
\end{equation*}

The Heisenberg-Gabor limit states that no function can be simultaneously concentrated in both time and frequency, and so $\forall f : ||\mathcal{B}\mathcal{D}f||<||f||$. However, one can look for nearly-invariant functions under this operator, functions for which $||\mathcal{B}\mathcal{D}f||/||f||$ is as close to 1 as possible. 

These are the eigenfunctions of the operator $\mathcal{B}\mathcal{D}$, i.\ e.\ the solutions of the integral equation
\begin{equation}
\label{eqn:pswfintequation}
\lambda\psi(x)=\int_{-T}^T \psi(s)\frac{\sin(2\pi\Omega(x-s))}{\pi(x-s)}ds.
\end{equation}

Slepian and collaborators showed that this equation is solvable only for select values of $\lambda$, a countably infinite set $1>\lambda_0>\lambda_1>\dots>0$. The corresponding eigenfunctions $\psi_i$ were named Prolate Spheroidal Wave functions. The naming stems from the curious observation that these functions are solutions to the spheroidal wave equation
\begin{equation}
\label{eqn:diffoperator}
\left(1-\frac{x^2}{T^2}\right)\frac{d^2\psi_i}{dx^2}-2x\frac{d\psi_i}{dx}-\left(2\pi\Omega T\right)^2x^2\psi_i=\theta_i\psi_i .
\end{equation}
This is a Sturm-Liouville equation with a set of unique eigenvalues $\dots >\theta_{i-1}>\theta_i>\theta_{i+1}>\dots$ corresponding to the functions $\psi_i$\cite{Slepian1978a}.


Since $\mathcal{B}$ and $\mathcal{D}$ are idempotent operators, it is convenient to consider the $\psi_i$ eigenfunctions of the Hermitian operator $\mathcal{B}\mathcal{D}\mathcal{B}$. 
Timelimiting both sides of (\ref{eqn:pswfintequation}), the timelimited functions $\hat{\psi}_i=\mathcal{D}\psi_i$ are the eigenfunctions of the Hermitian operator $\mathcal{D}\mathcal{B}\mathcal{D}$, with corresponding eigenvalues $\lambda_i$. The term Prolate Spheroidal Wave function is used for both the $\psi_i$ and the $\hat{\psi_i}$.

As eigenfunctions of a Hermitian operator, the $\psi_i$ and $\hat{\psi}_i$ are orthogonal
\begin{equation*}
\int_{-\infty}^{\infty}\psi_i(x)\psi_j(x)dx=\delta_{ij}\qquad\int_{-T}^{T}\hat{\psi}_i(x)\hat{\psi}_j(x)dx=\lambda_i\delta_{ij},
\end{equation*}
and they are complete in $PW_\Omega$ and $\mathcal{L}^2_{[-T,T]}$, respectively. The Prolate Spheroidal Wave functions thus form an orthogonal base for $PW_{\Omega}$, and the eigenvalue $\lambda_i$ represents the fraction of energy of $\psi_i$ contained in $[-T,T]$.

This leads to a straightforward approach to continuous bandlimited extrapolation. Let $f$ be a function segment in $\mathcal{L}^2_{[-T,T]}$. Then
\begin{equation}
g=\sum_i\frac{1}{\lambda_i}\langle \hat{\psi}_i,f\rangle_{[-T,T]} \psi_i
\label{eqn:bandlapprox}
\end{equation}
is a bandlimited function that agrees with $f$ in the interval due to the completeness of the $\hat{\psi_i}$ in $\mathcal{L}^2_{[-T,T]}$. Furthermore, when truncating the sum, the first terms have the largest eigenvalues and capture the most energy inside the interval.

Using both (\ref{eqn:pswfintequation}) and (\ref{eqn:diffoperator}), the PSWFs were further shown to have the following properties \cite{Slepian1978a}:
\begin{enumerate}[i]
%
%
%
\item The $\psi_i$ are eigenfunctions of the finite Fourier transform,
\begin{equation*}
\int_{-T}^T e^{\imath 2 \pi t \xi}\psi_n(t)dt=\imath^n\left(\frac{\lambda
    T}{\Omega} \right)^{1/2}\psi_n\left(\frac{\xi T}{\Omega}\right).
\end{equation*}

\item \label{prop:eigenvalue}The eigenvalues $\lambda_i$ cluster near 1 for low values of $i$, and decay exponentially after a set breakpoint
\begin{equation*}
\lambda_i \approx 1, i \ll 4\Omega T \quad \mbox{and} \quad \lambda_i\approx 0, i\gg 4\Omega T.
\end{equation*}
 The width of the region where $\lambda_i \in (\epsilon,1-\epsilon)$ grows as $\log(\Omega T)$.
%
\end{enumerate}
The first property illustrates the symmetry of time and frequency domain in this problem. The second property shows that a truncated approximation (\ref{eqn:bandlapprox}) will require $\approx4\Omega T$ terms. The increasing irregularity of the $\hat{\phi}_i$ coupled with the exponential decay of the $\lambda_i$ ensures that $g$ will converge rapidly after this breakpoint for a sufficiently regular $f$.
%
%
\subsection{Periodic Discrete Prolate Spheroidal Sequences}
\label{sec:pdpss}
There are several possible discretisations for PSWFs. The most well-known is the one proposed by Slepian, where the Fourier transform is replaced with the discrete-time Fourier transform
\begin{equation*}
G(f)=\sum_{n=-\infty}^\infty g[n]e^{-i2\pi f n},\qquad g[n]=\frac{1}{2\pi}\int_{-1/2}^{1/2}G(f)e^{i2\pi f n}df.
\end{equation*}
The resulting generalisations are frequency domain functions on $[-1/2,1/2]$ commonly referred to as discrete prolate spheroidal wave functions (DPSWFs), and infinite sequences in the time domain called discrete prolate spheroidal sequences (DPSSs).  The properties from section \ref{sec:pswf} carry over, apart from some loss of symmetry between time and frequency domains \cite{Slepian1978}.

Another approach was proposed independently by Jain and Ranganath \cite{jain1981} and Gr\"unbaum \cite{Grunbaum1981a}, and elaborated on in \cite{Xu1984}. They replace the Fourier transform with the DFT for sequences of length $L$:
\begin{equation*}
H_k=\sum_{n=0}^{L-1}h[n] e^{-i 2\pi k n /L}\qquad h[n]=\frac{1}{L}\sum_{k=0}^{L-1}H_k e^{i2\pi kn/L}.
\end{equation*}
The discrete analogues of the time- and bandlimiting operators are given by the matrices (\ref{eqn:ddef}) and (\ref{eqn:bdef}) from the discrete band limited extrapolation problem. Similar to the continuous PSWFs, the periodic discrete prolate spheroidal sequences (P-DPSS) $\phi_i$ are defined as the eigenvectors of the hermitian matrix operator $B_\bN D_\bM'D_\bM B_\bN$, and their limited versions $\hat{\phi_i}$ as the eigenvectors of $D_\bM B_\bN D_\bM$
\begin{equation}
B_\bN D_\bM'D_\bM B_\bN\phi_i=\dlambda_i\phi_i\qquad D_\bM B_\bN D_\bM'\hat{\phi_i}=\dlambda_i\hat{\phi_i}.
\label{eqn:pdpssdef}
\end{equation}
Since $D_\bM B_\bN D_\bM'$ is not of full rank when $\bM>\bN$, an additional
requirement that $\dlambda_i\neq0$ is imposed, following \cite{Xu1984}. Consequently, there are only $\{\min(\bN,\bM)\}$ P-DPSSs.

The $\phi_i$ properties are similar to those of PSWFs:
\begin{enumerate}[i]
\item If $\bM\geq\bN$, the $\phi_i$ are complete in the space of bandlimited sequences spanned by $B_\bN$.
\item If $\bM\leq \bN$, the $\hat{\phi}_i$ are complete in $\mathbb{R}^{\bM}$.
\item The $\phi_i$ are doubly orthogonal
\begin{equation*}
\phi_i\cdot\phi_j=\delta_{ij},\quad D\phi_i\cdot D\phi_j=\dlambda_i\delta_{ij}
\end{equation*}
\item The P-DPSSs are eigenvectors of the finite DFT, but with the roles of $\bM$ and $\bN$ interchanged
\begin{equation*}
D_\bN F\phi_{\bM,\bN,i}=\hat{\phi}_{\bN,\bM,i}.
\end{equation*}
Here, $F$ is the DFT matrix of size $L$.
\item Among sequences of length $L$, with frequency support in $[0,\sm]\cup[L-\sm,L]$, $\phi_0$ is the most concentrated in $[0,\bM-1]$. Among sequences of equally limited frequency support orthogonal to $\phi_0$, $\phi_1$ is the most concentrated in $[0,\bM-1]$, and so on.
\item Like the eigenvalues of the PSWFs, the eigenvalues $\dlambda_i$ are
  distinct and cluster exponentially near 1 and 0, in that
\begin{equation*}
\dlambda_i \approx 1, i \ll \frac{NM}{L} \quad \mbox{and} \quad \dlambda_i\approx 0, i\gg \frac{NM}{L}.\label{eqn:pdpsseigs}
\end{equation*}
 The width of the plunge region where $\dlambda_i \in (\epsilon,1-\epsilon)$ grows as $\log(NM/L)$ \cite{Wilson1987}.
\item The P-DPSSs satisfy a second order difference equation \cite{Xu1984}
\begin{equation}
\label{eqn:2ndorderdiff}b_{k}\hat{\phi_i}[k-1]+c_k\hat{\phi_i}[k]+b_{k+1}\hat{\phi_i}[k+1]=\theta_i\hat{\phi_i}[k],\qquad k=0,\dots,\bM-1,
\end{equation}
with coefficients
\begin{align}
\notag b_k&=\sin\left(\frac{\pi k}{L}\right)\sin\left(\frac{\pi (\bM-k)}{L}\right)\\
\notag c_k&=-\cos\left(\frac{\pi (2k+1-\bM)}{L}\right)\cos\left(\frac{\pi \bN}{L}\right).
\end{align}
\end{enumerate}

Equation  (\ref{eqn:pdpssdef}) clarifies their importance in both the discrete band limited extrapolation problem, and discrete Fourier extensions. From (\ref{eqn:aaisjj}) it is obvious that the eigenvectors of $\DM\DM'$ and $JJ'$ are the limited P-DPSSs $\hat{\phi_i}$. Regarding the singular value decompositions of $\DM$ and J, this determines the left singular vectors as $\hat{\phi_i}$ and the associated singular values as $\sqrt{\dlambda_i}$. For $J$, following
\begin{equation*}
	J'J=B_\bN' D_\bM' D_\bM B_\bN=B_\bN D_\bM B_\bN,
\end{equation*}
the right singular vectors are seen to be full P-DPSSs $\phi_i$. Using these relations, Jains SVD method results in an extrapolated sequence
\begin{equation*}
g[n]=\sum_{i=0}^{N-1} \frac{1}{\sqrt{\dlambda_i}}\left(\hat{\phi}_i\cdot f\right)\phi_i,
\end{equation*}
similar to the continuous extrapolation (\ref{eqn:bandlapprox}).

The discrete Fourier extension matrix $\DM$ differs from $J$ in the right singular vectors. Similar to (\ref{eqn:aaisjj})
\begin{align*}
\DM'\DM&=\left\{\sum_{\jm=-\sm}^\sm\exp{\imath\frac{(p-q)2\pi \jn}{L}}\right\}, \qquad p,q=-\sn,\dots,\sn\\
&=D_\bN B_\bM D_\bN.
\end{align*}
The right singular vectors are then again limited P-DPSSs $\hat{\phi_i}$, where the parameters $\bN$ and $\bM$ have been interchanged. 
Using $\hat{\Phi}$ to represent the set of P-DPSSs and $\dLambda$ for the diagonal matrix of eigenvalues, the SVD of $\DM$ is given by
\begin{equation*}
\DM=\hat{\Phi}_{\bM,\bN}\sqrt{\dLambda}\hat{\Phi}'_{\bN,\bM}.
\end{equation*}

Calculating the coefficients of the discrete FE by truncating the SVD now corresponds to
\begin{equation*}
a=\sum_{i=0}^{i_c} \frac{1}{\sqrt{\dlambda_i}}\left(\hat{\phi}_{\bN,\bM,i}\cdot f\right)\hat{\phi}_{\bM,\bN,i},
\end{equation*}
where $i_c$ is determined by the cutoff parameter,
\begin{equation*}
  \sqrt{\dlambda_{i_c}}\geq \cutoff > \sqrt{\dlambda_{i_c+1}}.
\end{equation*}
This formulation of the FE together with the P-DPSS properties listed above leads to the fast algorithms in the next section.


%% file: algorithms.tex

\section{Algorithms}
\label{sec:algorithms}
In this section we present two approaches to computing the truncated SVD
solution to the discrete Fourier extension problem $\DM x=b$. Both approaches
are based on a specific division in easier subproblems, which is
explained below. They differ only in how they solve one of these
subproblems. Detailed algorithms follow in sections~\ref{sec:tsvd} and~\ref{sec:pa}.

From the previous section, the SVD of the FE matrix $\DM$ is 
\begin{equation*}
\DM=U\Sigma V'=\hat{\Phi}_{\bM,\bN}\sqrt{\dLambda}\hat{\Phi}'_{\bN,\bM}.
\end{equation*}

The cost of this full SVD is prohibitively large, growing as ${\mathcal O}(N^3)$. To reach
the performance goal of ${\mathcal O}(N\log{N})$ operations, two subproblems are identified
and solved in succession.
The key to this subdivision is in the distribution of the singular
values $\Sigma$,  as shown in figure~\ref{fig:intervals}. Three distinct regions are
visible. Following property (\ref{eqn:pdpsseigs}) of the P-DPSS, they are for
some small cutoff $\cutoff$:
\begin{itemize}
\item A region $I_\alpha:=\{\sigma : 1>\sigma>1-\cutoff\}$ where all singular
  values are 1 up to a tolerance $\cutoff$. This region contains approximately ${\mathcal O}(NM/L)$ singular values.
\item A region $I_\beta:=\{\sigma : 1-\cutoff \geq \sigma > \cutoff\}$, also
  referred to as the ``plunge region'' in a more general context regarding
  truncated frames.
The number of singular values in this region grows as ${\mathcal O}(\log{(NM/L)})$.
\item A region $I_\gamma:=\{\sigma : \cutoff \geq \sigma > 0\}$ where the
  singular values further decay exponentially. 
\end{itemize}
\setlength{\figurewidth}{8cm}
\setlength{\figureheight}{4cm}
\begin{figure}[hbtp]
\centering
 \begin{tikzpicture}
      \node (A) at (0,0){
        \input{img/figure7.tex}
              };
      \node  (inter) at (3.6,-0.3) {$I_\gamma$
            };
      \node  (inter2) at (0.8,-0.3) {$I_{\beta}$
            };
      \node  (inter3) at (-2.0,-0.3) {$I_\alpha$
};
    \end{tikzpicture}
    \caption{The subdivision of the spectrum of $\DM$ into three distinct
      intervals, with cutoff parameter $\tau=1e-14$. Due to rounding errors, the eigenvalues in region $I_\gamma$ don't decay past machine precision.}
\label{fig:intervals}
\end{figure}
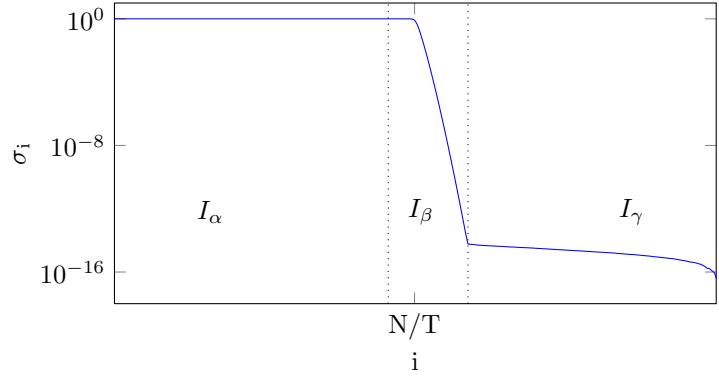

Let the subdivision of the singular values and associated vectors be denoted by
\begin{equation*}
\DM=\begin{bmatrix}U_\alpha & U_\beta & U_\gamma\end{bmatrix}\begin{bmatrix}\Sigma_\alpha & & \\ & \Sigma_\beta & \\ & & \Sigma_\gamma \end{bmatrix}\begin{bmatrix}V_\alpha &  V_\beta & V_\gamma\end{bmatrix}'.
\end{equation*}
The truncated SVD solution to the problem $\DM x=b$ with truncation cutoff at
$\cutoff$ is then
\begin{equation}
x=x_\alpha+x_\beta=V_\alpha\Sigma_\alpha^{-1}
U_\alpha'b_\alpha+V_\beta\Sigma_\beta^{-1} U_\beta'b_\beta,
\label{eqn:svdsubdivision}
\end{equation}
where the inverse operator applies to the diagonal elements. The right hand side
is split along the orthogonal spans of $U_\alpha, U_\beta$ and $U_\gamma$, i.e.,
\[
 b_\alpha= U_\alpha U_\alpha'b, 
\]
and so on. Note that this solution method implicitly assumes $b_\gamma$ to be below the
required solution accuracy. If it is not, a large solution term $x_\gamma$ is needed, with 
$||x_\gamma||\geq\frac{1}{\cutoff}||b_\gamma||$. 
In this case, we say the
Fourier extension has not yet converged, and $\bN$ should be increased.


Equation (\ref{eqn:svdsubdivision}) splits the problem into two orthogonal
subproblems. The isolation of $V_\beta, S_\beta$ and $ U_\beta$ from the plunge region and
 efficient calculation of $x_\beta$ is where the two approaches given in sections~\ref{sec:tsvd} and~\ref{sec:pa}
differ. 

The subsequent calculation of $x_\alpha$ is then straightforward,
based on the following observation:
\begin{align*}
  A'(b-b_\beta)&=V_\alpha\Sigma_\alpha U_\alpha' b_\alpha+V_\gamma\Sigma_\gamma
                 U_\gamma'b_\gamma\\
               &=V_\alpha\Sigma_\alpha^{-1} U_\alpha'b_\alpha+{\mathcal O}(\tau)\\
               &=x_\alpha+{\mathcal O}(\tau)
\end{align*}
The $b_\gamma$ term, which is already assumed to be negligible, is fully
eliminated by the additional ${\mathcal O}(\cutoff)$ factor $\Sigma_\gamma$. Noting that
$\Sigma_\alpha=\Sigma_\alpha^{-1}+{\mathcal O}(\cutoff)$ then yields $x_\alpha$ at the cost
of a single fast multiplication with $A'$ (${\mathcal O}(N\log{N})$).

\subsection{Approach 1: explicit projection onto prolate sequences}
\label{sec:tsvd}
Due to the intrinsic connection of the FE problem with DPSSs, it is possible to
explicitly compute  $U_\beta, \Sigma_\beta$ and $ V_\beta'$. The second order difference equation from (\ref{eqn:2ndorderdiff}) implies that

\begin{equation}
T_{\bM,\bN}\hat{\phi}_{\bM,\bN,i}=\theta_i\hat{\phi}_{\bM,\bN,i},\qquad T_{\bM,\bN}=\begin{bmatrix} c_0 & b_1 & & \\
b_1 & c_1 & \ddots & \\
& \ddots & \ddots & b_{N-1} \\
& & b_{N-1} & c_{N-1}  \end{bmatrix}.
\label{eqn:tridiag}
\end{equation}

The P-DPSSs can thus be found as eigenvectors of a tridiagonal matrix. The
${\hat{\phi}_{\bM,\bN,i}}$ that make up $U$ are eigenvectors of
$T_{\bM,\bN}$. The dual matrix $T_{\bN,\bM}$ yields the right singular vectors
$V$. With the P-DPSSs known, the original singular value is found in ${\mathcal O}(N\log{N})$
operations as  
\begin{align*}
\sigma_i=\hat{\phi}_{m,n,i}'\DM\hat{\phi}_{n,m,i}
\end{align*}
This approach is already in use for the regular DPSSs~\cite{Prolate1994}. 


In this stage of the algorithm however, we are only interested in a subset of
the singular values and vectors. Since the number of singular values in the plunge region is logarithmically small, we look to algorithms that
calculate $k$ eigenvalues and eigenvectors of a tridiagonal matrix $T$ in
${\mathcal O}(kN)$ operations~\cite{Dhillon1997}. Algorithms for this computation require as input for the desired eigenvalues $\theta_i$ of $T$ either:
\begin{itemize}
\item A range $[C_1,C_2]$: the algorithm then finds all $\theta : C_1\leq \theta \leq C_2$.
\item A index set $\{i_{\min},\dots,i_{\max}\}$: the algorithm then finds all $\theta : \theta_{i_{\min}} \leq \theta \leq \theta_{i_{\max}}$, from the ordered set $\theta_0\geq \theta_1\geq\dots$.
\end{itemize}

The algorithms thus require knowing the $\theta_j$, or alternatively
the indices $j$, that correspond to ${\sigma_i}\in I_\beta$. Denote the mappings between $\sigma_i$ and $\theta_j$, and between their indices $i$ and $j$, by
\begin{equation*}
\theta_j=\mathcal{G}_{\bN,\bM}(\sigma_i), \quad j=\mathcal{G}_{\bN,\bM}'(i)\Leftrightarrow\begin{cases}\hat{\phi}_i'T_{\bN,\bM}\hat{\phi}_i=\theta_j&\\
\hat{\phi}_i'\DM_{\bN,\bM}\hat{\phi}_i=\sigma_i.
\end{cases}\label{eqn:mappings}
\end{equation*}
The description of $\mathcal{G}_{\bN,\bM}(I_\beta)$ is difficult, as very little
is known of this mapping. The only known trait is monotonicity, a trait shared with the PSWF and DPSWF equivalents. This is already very helpful, since it implies that $\mathcal{G}_{\bN,\bM}'(i) = i$.

\begin{thm} For $\bM \geq \bN$, the mapping $\mathcal{G}_{\bN,\bM}'$ is monotone,
\begin{equation*}
\forall i_1, i_2 : i_1>i_2 \Leftrightarrow \mathcal{G}_{\bN,\bM}'(i_1)>\mathcal{G}_{\bN,\bM}'(i_2).
\end{equation*}
\label{thm:mapping}
\end{thm}
\begin{proof}
The proof follows a mechanism used by Slepian in both \cite[p. 61]{Slepian1978a} and
\cite[\S4.1]{Slepian1978}. The continuity of eigenvalues and eigenvectors as a function of a parameter, combined with a known ordering result for a specific value of this parameter,  extends the known result to all parameter values.

First we recall a similar result from the continuous Fourier Extension. Let $\CM$
be the matrix of the continuous FE (\ref{eqn:continuousfe}) with eigenvalues $\bar{\sigma}_i$, and
$\bar{T}$ the tridiagonaln matrix (\ref{eqn:barTtridiag}) with eigenvalues
$\bar{\theta}_i$
\begin{align}
\bar{T}&=\begin{bmatrix}p
\bar{c}_0 & \bar{b}_1 & & \\
\bar{b}_1 & \bar{c}_1 & \ddots & \\
& \ddots & \ddots & \bar{b}_{N-1} \\
& & \bar{b}_{N-1} & \bar{c}_{N-1}
\end{bmatrix},\label{eqn:barTtridiag}
\qquad\bar{c}_i=\left(\frac{N-1}{2}-i\right)^2\cos\frac{\pi}{T},\qquad\bar{b}_i=\frac{i(N-i)}{2}.
\end{align}

 Then the mapping $\bar{\mathcal{G}}'$
\begin{equation*}
j=\bar{\mathcal{G}}'(i)\Leftrightarrow\begin{cases}\bar{\phi}_i'\bar{T}\bar{\phi}_i=\bar{\theta}_j&\\
\bar{\phi}_i'\CM\bar{\phi}_i=\bar{\sigma}_i & \end{cases},
\end{equation*}
was proven to be monotone in \cite[\S4.1]{Slepian1978a}. This result can be
extended to
the discrete FE case, since the limits of the entries of $\DM'\DM$ and $T_{N,M}$ for large $\sm$ can be written in terms of the corresponding matrices of the continuous problem:
\begin{align*}
\lim_{\sm\to\infty}(\DM'\DM)_{ij}&=\frac{\sin{\frac{(i-j)\pi}{T}}}{\pi(i-j)}=\CM_{ij}\\
\lim_{\sm\to\infty}b_k&=\frac{\pi^2k(N-k)}{L^2}=\frac{2\pi^2}{L^2}\bar{b}_k\\
\lim_{\sm\to\infty}c_k&=\cos{\frac{\pi}{T}}\left(\frac{2\pi^2}{L^2}\left(\frac{N-1}{2}-k\right)^2-1\right)=\frac{2\pi^2}{L^2}\bar{c}_k-\cos{\frac{\pi}{T}}\\
\lim_{\sm\to\infty}T_{\bN,\bM}&=\frac{2\pi^2}{L^2}\bar{T}-\cos{\frac{\pi}{T}}I
\end{align*}
The last line ensures that the eigenvalues of $T_{\bN,\bM}$ are those of
$\bar{T}$ under a linear mapping.
Since this mapping preserves the ordering of eigenvalues, the theorem holds in the limit $\sm\to\infty$. 
Further, when limited to the regime $T>1,2\sm+1\geq\bN$, we have the following:
\begin{enumerate}
\item The tridiagonal matrix $T_{\bM,\bN}(\sm)$ with diagonal elements $c_k$ and
  subdiagonal elements $b_k$ 
\begin{align*}
\notag b_k&=\sin\left(\frac{\pi k}{2T\sm}\right)\sin\left(\frac{\pi (2\sm+1-k)}{2T\sm}\right)\\
\notag c_k&=-\cos\left(\frac{\pi (k-\sm)}{T\sm}\right)\cos\left(\frac{\pi \bN}{2T\sm}\right).
\end{align*}
commutes with $\DM'\DM$ for integer values of $\sm$. However, by a substitution
as in \cite[Thm. 4.7]{Chamzas1980a}, it is easy to see that this relationship holds for
any real $\sm>\sn$.

\item  A classical result states that the eigenvalues of a matrix are
continuous as the matrix entries vary continuously in the
parameter. Thus all $\theta_i(\sm)$ are continuous. In general they may coincide with each other. However, since all subdiagonal entries are always non-zero,
\begin{equation*}
\forall 1\leq k \leq \bN-1: b_k^2 > 0,
\end{equation*}
 $T_{\bM,\bN}(\sm)$ is a so-called normal Jacobi matrix and such matrices are
 known to have distinct eigenvalues \cite[Ch. 2.1]{gantmakher1961}. As a result of this
 distinctness, the eigenvectors can be chosen to be continuous in $\sm$ as well
 \cite[Ch. 2 \S5.3]{kato2012perturbation}.

\item To prove similar statements for the matrix $\DM'\DM$, we note that the
  inductive proof in \cite[Prop. 5]{Xu1984} for the distinctness of eigenvalues of $\DM'\DM$ is
  dependent only on it being symmetric, and the commuting $T_{\bM,\bN}(\sm)$
  being a normal Jacobi matrix. The eigenvalues $\lambda_i(\sm)$ are thus continuous and distinct $\forall\sm\geq\sn$. Then the eigenvectors can also be chosen continuous in $\sm$.
\end{enumerate}
Combining these statements, the distinctness preserves the relative ordering of
the continuous eigenvalues. The continuity of the eigenvectors relates the
eigenvalues of $\DM$ and $T_{\bN,\bM}$ through the mapping $\mathcal{G}'$. Since
this mapping is monotone in the limit $\sm\to\infty$, the continuity ensures the
mapping is monotone for all $\sm\geq\sn$.
\end{proof}

The required index set for the eigenvalues of $T_{\bM,\bN}$ is exactly the index set
corresponding to $\sigma_i$ from the plunge region. From (\ref{eqn:pdpsseigs})
these are known to be centered around $\frac{\bN\bM}{L}$, and their number grows
as ${\mathcal O}(\log{\bN})$. All that remains is to determine the constants $C_{\min{}}$ and
$C_{\max{}}$ so that 
\begin{equation*}
\sigma_j\in I_\beta\Leftrightarrow j\in\left\{N-\frac{\bN\bM}{L}-C_{\min}\log{\bN},N-\frac{\bN\bM}{L}+C_{\max}\log{\bN}\right\}.
\end{equation*}
The minimum required index set with cutoff $\cutoff=1e-16$ for increasing $\bN$ is shown in Figure~\ref{fig:Tnlog}, with the value $\bN-\frac{\bN\bM}{L}$ as a dashed line. Experimentally, the choices $C_{\min}\geq6$ and $C_{\max}\geq3$ seem sufficient for all $\cutoff\geq\epsilon_{mach}$. The $V_\beta$ can be obtained by refining $\DM'\hat{\phi_{\bM,\bN,i}}$ as eigenvectors of $T_{\bN,\bM}$.

Using a fast tridiagonal eigenvector algorithm, $U_\beta$, $\Sigma_\beta$ and
$V_\beta$ can be computed in ${\mathcal O}(N\log{N})$ operations. The solution term

\begin{equation*}
x_\beta=V_\beta\Sigma_\beta^{-1} U_\beta'b
\end{equation*}
is then found in ${\mathcal O}(N\log^2{N})$ operations.

\begin{remark}
Monotonicity of the map $\mathcal{G}_{\bN,\bM}$ is also observed to hold for integer $\bM$ smaller than integer $\bN$. Specifically, a variant of Theorem~\ref{thm:mapping} holds for the $\bM$ nonzero singular values $\sigma_i$. The same index set can thus be used to compute both $U_\beta$ and $V_\beta$ with a fast tridiagonal eigenvalue algorithm. However, this altered algorithm is without proof.
\end{remark}
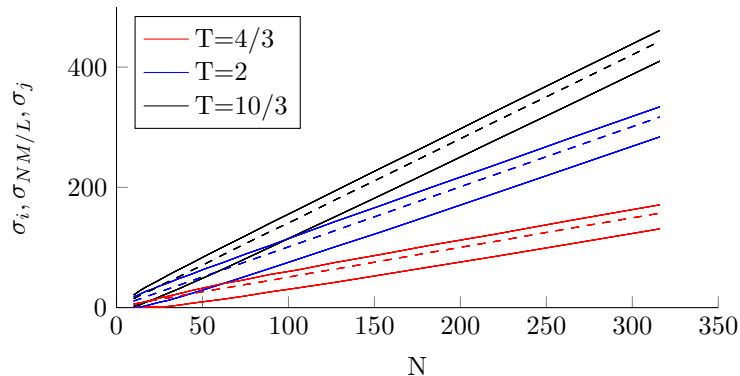
\begin{figure}[htbp]
\begin{center}
\input{img/Tnlogregion.tex}
\caption{The behaviour of the index set of the plunge region. The minimal and maximal index of the plunge region are shown as solid lines, for different values of $T$. The point $NM/L$, which is known to lie in the interval, is shown as a dashed line.}
\label{fig:Tnlog}
\end{center}
\end{figure}

Combining $x_\beta$ with the calculation of $x_\alpha$ described above leads to
a fast ${\mathcal O}(N\log^2{N})$ algorithm. A bare-bones version is given in Algorithm~\ref{alg:commutation}.

\begin{algorithm}[htbp]
\begin{algorithmic}
\STATE $U_\beta$, $V_\beta$ =eig$(T_N,T_M,\{j_{\min{}} ,j_{\max{}} \})$
\STATE $\Sigma_\beta=(U_\beta'\DM V_\beta)>\cutoff$
\STATE $x_\beta=V_\beta\Sigma_\beta^{-1}U_\beta'b$
\STATE $x_\alpha=A'(b-Ax_\beta)$
\STATE $x=x_\alpha+x_\beta$
\end{algorithmic}
\caption{Explicit projection on DPSSs}
\label{alg:commutation}
\end{algorithm}

\subsection{Approach 2: an implicit projection method}
\label{sec:pa}
The second approach to calculating $x_\beta$ is more universal, or more
generally applicable, since it depends solely on the steep singular value
profile discussed earlier and illustrated in Fig.~\ref{fig:intervals}. As such,
it is extensible to any ill-conditioned basis that exhibits a similar profile. 

The approach is based on the observation that multiplying both the FE matrix $\DM$ and right hand side by a factor 
\begin{equation}\label{eq:P}
P=(\DM\DM'-I)
\end{equation}
isolates the problem to the plunge region. This is most easily seen from the SVD. With $\DM=U\Sigma V'$ we have
\begin{align}
  P&=U(\Sigma^2-I) U',
\label{eqn:svP}
\end{align}
and
\begin{equation*}
  P\DM= U(\Sigma^3-\Sigma)V'.
\end{equation*}
Note here that the mapping
\begin{equation*} 
\mathcal{P}(\sigma)=\sigma^3-\sigma
\end{equation*} 
isolates the singular values from the plunge region since $\forall \sigma
\in \{I_\alpha \cup I_\gamma\}: \mathcal{P}(\sigma)={\mathcal O}(\cutoff)$. This way,
$P\DM$ preserves the singular vectors of just the plunge region, but
with mapped singular values
\begin{equation}
\label{eqn:PDMapprox}
P\DM = U_\beta(\Sigma_\beta^3-\Sigma_\beta)V_\beta'+{\mathcal O}(\cutoff).
\end{equation}

In theory, $P$ is a square full rank matrix, and solving
\begin{equation}
\label{eqn:projected_equation}
P\DM x=Pb
\end{equation}
is equivalent to solving $\DM x=b$. In practice, $P\DM$ has a
large numerical nullspace. Hence, $P \DM$ is approximately low rank, with the rank increasing with the size of the plunge region. The combination of this low rank with a fast matrix-vector product allows random matrix algorithms to solve \eqref{eqn:projected_equation} very efficiently.

Assume $W$ a uniform random matrix of dimensions $\bN\times R$, where
$R=C\log{N}+D$ is a conservative estimate for the rank of $P\DM$. From the
previous section, $C=C_{\min}+C_{\max}\geq 9$ is sufficient, with $D\sim 10$
minimizing the chance of failure of the random matrix algorithm. The column space
of $P\DM W$ then approximates the column space of $P\DM$ very well.

Therefore, solving the following small linear system 
\[
P\DM Wy =Pb
\]
and letting
\[
  x_W =Wy
\]
one obtains a solution to \eqref{eqn:projected_equation} a cost of ${\mathcal O}(\bM R^2)$. It follows from (\ref{eqn:svP}) and (\ref{eqn:PDMapprox}) that $x_\beta$ is recovered exactly. On the other hand, this solution process introduces additional solution terms from the nullspace of $PA$. Write $x_W$ as $x_\beta+r_\alpha+r_\gamma$. Then as before we calculate
\begin{align*}
\DM'(b-\DM x_W)  &=\DM'(b_\alpha-\DM r_\alpha-\DM r_\gamma)+{\mathcal O}(\cutoff)\\
                 &=x_\alpha -r_\alpha-V_\gamma \Sigma_\gamma^2 V_\gamma' r_\gamma
                   +{\mathcal O}(\cutoff) \\
                 &=x_\alpha - r_\alpha + {\mathcal O}(\cutoff).
\end{align*}
Then $x =  x_W+\DM'(b-\DM x_W)$ boils down to
\begin{align*}
x=x_\alpha+x_\beta+r_\gamma+{\mathcal O}(\cutoff).
\end{align*}
so that $||\DM x-b||={\mathcal O}(\cutoff)$. The total cost of this algorithm is again
${\mathcal O}(N\log^2{N})$ operations. A step by step version is given in Algorithm~\ref{alg:dpss2}.



\begin{algorithm}[htbp]
\begin{algorithmic}
\STATE $PAWy=Pb$
\STATE $x_\beta=Wy$
\STATE $x_\alpha=\DM'(b-\DM x_\beta)$
\STATE $x=x_\alpha+x_\beta$
\end{algorithmic}
\caption{Implicit projection on DPSSs}
\label{alg:dpss2}
\end{algorithm}

\subsection{Adaptation for the continuous FE}
\label{sec:continuousalg}
As mentioned in section \ref{sec:pdpss}, the continuous FE also has numerous connections
with Prolate Spheroidal Wave theory. The solution is most easily expressed in terms of 
the eigenvectors of $\CM$, the DPSSs
\begin{equation*}
  \CM=\Psi\Lambda\Psi.
\end{equation*}

The eigenvalues $\lambda_i$ have a similar profile to that of Figure
\ref{fig:intervals}. Our second approach thus applies immediately, and Algorithm
\ref{alg:dpss2} is well suited to solve the continuous FE problem. Note however
that this does not eliminate the theoretical ${\mathcal O}(\sqrt{\cutoff})$ error bound.

Furthermore, these DPSS also satisfy a second order difference equation, different from
(\ref{eqn:2ndorderdiff}). In particular, the matrix $\CM$ commutes with the
tridiagonal matrix (\ref{eqn:barTtridiag}).

It follows that, with minor modifications, Algorithm~\ref{alg:commutation} can also be used to solve the continuous FE problem.


%% file: img/figure7.tex
%
%
\begin{tikzpicture}

\begin{axis}[%
width=\figurewidth,
height=\figureheight,
scale only axis,
xmin=1,
xmax=401,
xtick={200.5},
xticklabels={N/T},
xlabel={i},
ymode=log,
ymin=1e-18,
ymax=10,
ytick={1e-16, 1e-08,     1},
yminorticks=true,
ylabel={$\sigma{}_\text{i}$}
]
\addplot [
color=blue,
solid,
forget plot
]
table[row sep=crcr]{
1 1\\
2 1\\
3 1\\
4 1\\
5 1\\
6 1\\
7 1\\
8 1\\
9 1\\
10 1\\
11 1\\
12 1\\
13 1\\
14 1\\
15 1\\
16 1\\
17 1\\
18 1\\
19 1\\
20 1\\
21 1\\
22 1\\
23 1\\
24 1\\
25 1\\
26 1\\
27 1\\
28 1\\
29 1\\
30 1\\
31 1\\
32 1\\
33 1\\
34 1\\
35 1\\
36 1\\
37 1\\
38 1\\
39 1\\
40 1\\
41 1\\
42 1\\
43 1\\
44 1\\
45 1\\
46 1\\
47 1\\
48 1\\
49 1\\
50 1\\
51 1\\
52 1\\
53 1\\
54 1\\
55 1\\
56 1\\
57 1\\
58 1\\
59 1\\
60 1\\
61 1\\
62 1\\
63 1\\
64 1\\
65 1\\
66 1\\
67 1\\
68 1\\
69 1\\
70 1\\
71 1\\
72 1\\
73 1\\
74 1\\
75 1\\
76 1\\
77 1\\
78 1\\
79 1\\
80 1\\
81 1\\
82 1\\
83 1\\
84 1\\
85 1\\
86 1\\
87 1\\
88 1\\
89 1\\
90 1\\
91 1\\
92 1\\
93 1\\
94 1\\
95 1\\
96 1\\
97 1\\
98 1\\
99 1\\
100 1\\
101 1\\
102 1\\
103 1\\
104 1\\
105 1\\
106 1\\
107 1\\
108 1\\
109 1\\
110 1\\
111 1\\
112 1\\
113 1\\
114 1\\
115 1\\
116 1\\
117 1\\
118 1\\
119 1\\
120 1\\
121 1\\
122 1\\
123 1\\
124 1\\
125 1\\
126 1\\
127 1\\
128 1\\
129 1\\
130 1\\
131 1\\
132 1\\
133 1\\
134 1\\
135 1\\
136 1\\
137 1\\
138 1\\
139 1\\
140 1\\
141 1\\
142 1\\
143 1\\
144 1\\
145 1\\
146 1\\
147 1\\
148 1\\
149 1\\
150 1\\
151 1\\
152 0.999999999999999\\
153 0.999999999999999\\
154 0.999999999999999\\
155 0.999999999999999\\
156 0.999999999999999\\
157 0.999999999999999\\
158 0.999999999999999\\
159 0.999999999999999\\
160 0.999999999999999\\
161 0.999999999999999\\
162 0.999999999999999\\
163 0.999999999999999\\
164 0.999999999999999\\
165 0.999999999999999\\
166 0.999999999999999\\
167 0.999999999999999\\
168 0.999999999999999\\
169 0.999999999999999\\
170 0.999999999999999\\
171 0.999999999999999\\
172 0.999999999999999\\
173 0.999999999999999\\
174 0.999999999999999\\
175 0.999999999999999\\
176 0.999999999999999\\
177 0.999999999999999\\
178 0.999999999999998\\
179 0.999999999999998\\
180 0.999999999999996\\
181 0.999999999999996\\
182 0.999999999999994\\
183 0.999999999999919\\
184 0.999999999999423\\
185 0.999999999996074\\
186 0.999999999973983\\
187 0.999999999832426\\
188 0.999999998952047\\
189 0.999999993644962\\
190 0.999999962680517\\
191 0.999999788119967\\
192 0.999998839238557\\
193 0.999993878273071\\
194 0.999969011704913\\
195 0.999850030505143\\
196 0.999310095837935\\
197 0.997012128074812\\
198 0.988042423955423\\
199 0.95739991677528\\
200 0.87329565480279\\
201 0.707106781186548\\
202 0.487190619062562\\
203 0.288765301514372\\
204 0.154182257294063\\
205 0.0772451711871729\\
206 0.037139363973817\\
207 0.017318097437692\\
208 0.00787246021866295\\
209 0.00349905935643662\\
210 0.00152365400785586\\
211 0.00065096852348711\\
212 0.000273201327214996\\
213 0.00011273896632449\\
214 4.57810880511899e-05\\
215 1.83071315633366e-05\\
216 7.21336233455454e-06\\
217 2.80199418265696e-06\\
218 1.07352223327795e-06\\
219 4.05834968634566e-07\\
220 1.51441266313666e-07\\
221 5.58009126251093e-08\\
222 2.03082761446577e-08\\
223 7.30230801877191e-09\\
224 2.5948541884616e-09\\
225 9.11449414820242e-10\\
226 3.16528675706672e-10\\
227 1.08702578736242e-10\\
228 3.69232002268287e-11\\
229 1.24066728685844e-11\\
230 4.12463691980625e-12\\
231 1.35733582145616e-12\\
232 4.42003964741545e-13\\
233 1.42742163701434e-13\\
234 4.54773519809507e-14\\
235 1.50687724205528e-14\\
236 5.89983174337997e-15\\
237 5.69881624950702e-15\\
238 5.6003323309543e-15\\
239 5.50264185774455e-15\\
240 5.4252600360708e-15\\
241 5.1830735312815e-15\\
242 5.1368879205954e-15\\
243 4.97629386919343e-15\\
244 4.92401427098856e-15\\
245 4.84999617010169e-15\\
246 4.79805998390163e-15\\
247 4.70579925112866e-15\\
248 4.66657367563156e-15\\
249 4.57186165188267e-15\\
250 4.56419158349445e-15\\
251 4.51714969728416e-15\\
252 4.44897945565394e-15\\
253 4.41297696712668e-15\\
254 4.36775913047172e-15\\
255 4.28796601023895e-15\\
256 4.26883241550268e-15\\
257 4.20198640858813e-15\\
258 4.16066586411662e-15\\
259 4.14502491942751e-15\\
260 4.04752773071054e-15\\
261 4.02002070958001e-15\\
262 3.94604127047529e-15\\
263 3.87563009433251e-15\\
264 3.83578980667469e-15\\
265 3.80666507640145e-15\\
266 3.7969027820344e-15\\
267 3.7424793338949e-15\\
268 3.70087844460706e-15\\
269 3.61869239789062e-15\\
270 3.6006269263034e-15\\
271 3.55097631364902e-15\\
272 3.52239227413196e-15\\
273 3.4869403117794e-15\\
274 3.46278563257072e-15\\
275 3.42369037991773e-15\\
276 3.38107814020415e-15\\
277 3.34956028405386e-15\\
278 3.30421414059321e-15\\
279 3.28383667685452e-15\\
280 3.21840225752796e-15\\
281 3.15067051499966e-15\\
282 3.13252297187441e-15\\
283 3.08737594026692e-15\\
284 3.04780961496623e-15\\
285 3.00818291527548e-15\\
286 2.95638169015464e-15\\
287 2.93489875108495e-15\\
288 2.90476027804772e-15\\
289 2.85911180214696e-15\\
290 2.80235856811339e-15\\
291 2.78538880670818e-15\\
292 2.71932185266321e-15\\
293 2.7028149862441e-15\\
294 2.66414753452827e-15\\
295 2.62919250445293e-15\\
296 2.59251049657372e-15\\
297 2.56699608402979e-15\\
298 2.54487374527012e-15\\
299 2.50489330502225e-15\\
300 2.46312331986186e-15\\
301 2.44142121999655e-15\\
302 2.40276710329481e-15\\
303 2.37509993707391e-15\\
304 2.32760359772636e-15\\
305 2.3057109899816e-15\\
306 2.27921314941029e-15\\
307 2.2432645452491e-15\\
308 2.22943989892751e-15\\
309 2.18497200067638e-15\\
310 2.1803415268201e-15\\
311 2.14670573379456e-15\\
312 2.1116787442456e-15\\
313 2.08489616763686e-15\\
314 2.04828803079519e-15\\
315 2.02093026403333e-15\\
316 2.00257062852027e-15\\
317 1.97577068711085e-15\\
318 1.94066631808673e-15\\
319 1.91035337961431e-15\\
320 1.90261049242461e-15\\
321 1.84960496418279e-15\\
322 1.83352361982122e-15\\
323 1.81591373010123e-15\\
324 1.78676907654867e-15\\
325 1.76868212275333e-15\\
326 1.74769066088856e-15\\
327 1.72188782568121e-15\\
328 1.70335637345011e-15\\
329 1.67396320575549e-15\\
330 1.6646921689456e-15\\
331 1.62479841653213e-15\\
332 1.60095198002813e-15\\
333 1.58228863842419e-15\\
334 1.57204319067562e-15\\
335 1.51027929859025e-15\\
336 1.48555755613872e-15\\
337 1.47923834025061e-15\\
338 1.42468360971492e-15\\
339 1.40876547145308e-15\\
340 1.38621171940967e-15\\
341 1.35182593207901e-15\\
342 1.3338828805722e-15\\
343 1.30401525627695e-15\\
344 1.28700599005213e-15\\
345 1.27004371995624e-15\\
346 1.26650795571463e-15\\
347 1.23584287166474e-15\\
348 1.20799989452959e-15\\
349 1.191364311687e-15\\
350 1.17656238002719e-15\\
351 1.1372878772812e-15\\
352 1.1050297518822e-15\\
353 1.09231046378957e-15\\
354 1.07630217583652e-15\\
355 1.04336872251893e-15\\
356 1.02879715661698e-15\\
357 1.00929804473092e-15\\
358 9.7495379792071e-16\\
359 9.37478937407565e-16\\
360 9.23634426080908e-16\\
361 9.05752765255355e-16\\
362 8.80052494521279e-16\\
363 8.64417961459509e-16\\
364 8.44562774165038e-16\\
365 8.30659505899575e-16\\
366 8.01401604611504e-16\\
367 7.84253680198864e-16\\
368 7.59416611834546e-16\\
369 7.47831645377325e-16\\
370 7.22365379396091e-16\\
371 7.06698251958441e-16\\
372 6.96040288326628e-16\\
373 6.74042645951504e-16\\
374 6.44772161945835e-16\\
375 6.31372938541487e-16\\
376 5.94889800459994e-16\\
377 5.80385176035403e-16\\
378 5.5516643360077e-16\\
379 5.31190354441509e-16\\
380 5.15445612551447e-16\\
381 4.95398957503591e-16\\
382 4.71427138202624e-16\\
383 4.42797777047703e-16\\
384 4.25667544405674e-16\\
385 4.14003375747016e-16\\
386 4.04113827858262e-16\\
387 3.81093296424355e-16\\
388 3.64146415369817e-16\\
389 3.43335810601689e-16\\
390 3.15687630730482e-16\\
391 3.04132458394402e-16\\
392 2.61932606754952e-16\\
393 2.38594712011038e-16\\
394 2.03790083156174e-16\\
395 1.71738790919171e-16\\
396 1.69245914299688e-16\\
397 1.51051096675349e-16\\
398 1.22201371100045e-16\\
399 1.06099499080366e-16\\
400 8.16167269072216e-17\\
401 3.95619830132959e-17\\
};
\addplot [
color=black,
dotted,
forget plot
]
table[row sep=crcr]{
183 1e-18\\
183 10\\
};
\addplot [
color=black,
dotted,
forget plot
]
table[row sep=crcr]{
236 1e-18\\
236 10\\
};
\end{axis}
\end{tikzpicture}%

%% file: img/Tnlogregion.tex
%
%
\begin{tikzpicture}

\begin{axis}[%
width=\figurewidth,
height=\figureheight,
scale only axis,
xmin=0,
xmax=350,
xlabel={N},
ymin=0,
ymax=500,
ylabel={\tnlogregion},
axis x line*=bottom,
axis y line*=left,
legend style={at={(0.03,0.97)},anchor=north west,draw=black,fill=white,legend cell align=left}
]
\addplot [
color=red,
solid
]
table[row sep=crcr]{
10 1\\
12 1\\
14 1\\
17 1\\
21 1\\
25 1\\
30 2\\
36 4\\
43 7\\
51 10\\
62 14\\
74 19\\
89 26\\
106 33\\
127 42\\
153 54\\
183 68\\
220 85\\
264 106\\
316 131\\
};
\addlegendentry{T=4/3};

\addplot [
color=blue,
solid
]
table[row sep=crcr]{
10 1\\
12 1\\
14 1\\
17 3\\
21 6\\
25 9\\
30 12\\
36 17\\
43 23\\
51 30\\
62 40\\
74 51\\
89 65\\
106 81\\
127 101\\
153 125\\
183 154\\
220 190\\
264 233\\
316 284\\
};
\addlegendentry{T=2};

\addplot [
color=black,
solid
]
table[row sep=crcr]{
10 2\\
12 4\\
14 5\\
17 9\\
21 13\\
25 18\\
30 24\\
36 31\\
43 40\\
51 50\\
62 64\\
74 80\\
89 100\\
106 123\\
127 151\\
153 186\\
183 227\\
220 278\\
264 338\\
316 410\\
};
\addlegendentry{T=10/3};

\addplot [
color=red,
dashed,
forget plot
]
table[row sep=crcr]{
10 6\\
12 7\\
14 8\\
17 9\\
21 12\\
25 14\\
30 16\\
36 19\\
43 23\\
51 27\\
62 32\\
74 38\\
89 45\\
106 54\\
127 64\\
153 77\\
183 92\\
220 110\\
264 132\\
316 157\\
};
\addplot [
color=red,
solid,
forget plot
]
table[row sep=crcr]{
10 6\\
12 7\\
14 8\\
17 10\\
21 13\\
25 16\\
30 19\\
36 23\\
43 28\\
51 33\\
62 40\\
74 46\\
89 55\\
106 63\\
127 75\\
153 88\\
183 104\\
220 122\\
264 145\\
316 171\\
};
\addplot [
color=blue,
dashed,
forget plot
]
table[row sep=crcr]{
10 11\\
12 13\\
14 15\\
17 18\\
21 22\\
25 26\\
30 31\\
36 37\\
43 44\\
51 52\\
62 63\\
74 75\\
89 90\\
106 107\\
127 128\\
153 154\\
183 184\\
220 221\\
264 265\\
316 317\\
};
\addplot [
color=blue,
solid,
forget plot
]
table[row sep=crcr]{
10 18\\
12 20\\
14 23\\
17 27\\
21 31\\
25 36\\
30 41\\
36 48\\
43 55\\
51 64\\
62 75\\
74 88\\
89 103\\
106 121\\
127 143\\
153 169\\
183 200\\
220 237\\
264 282\\
316 334\\
};
\addplot [
color=black,
dashed,
forget plot
]
table[row sep=crcr]{
10 15\\
12 18\\
14 21\\
17 25\\
21 30\\
25 36\\
30 43\\
36 51\\
43 61\\
51 72\\
62 88\\
74 105\\
89 126\\
106 149\\
127 179\\
153 215\\
183 257\\
220 309\\
264 370\\
316 443\\
};
\addplot [
color=black,
solid,
forget plot
]
table[row sep=crcr]{
10 21\\
12 25\\
14 29\\
17 34\\
21 40\\
25 46\\
30 54\\
36 63\\
43 73\\
51 85\\
62 101\\
74 118\\
89 140\\
106 164\\
127 194\\
153 231\\
183 273\\
220 326\\
264 388\\
316 461\\
};
\addplot [
color=red,
solid,
forget plot
]
table[row sep=crcr]{
10 1\\
12 1\\
14 1\\
17 1\\
21 1\\
25 1\\
30 2\\
36 4\\
43 7\\
51 10\\
62 14\\
74 19\\
89 26\\
106 33\\
127 42\\
153 54\\
183 68\\
220 85\\
264 106\\
316 131\\
};
\addplot [
color=blue,
solid,
forget plot
]
table[row sep=crcr]{
10 1\\
12 1\\
14 1\\
17 3\\
21 6\\
25 9\\
30 12\\
36 17\\
43 23\\
51 30\\
62 40\\
74 51\\
89 65\\
106 81\\
127 101\\
153 125\\
183 154\\
220 190\\
264 233\\
316 284\\
};
\addplot [
color=black,
solid,
forget plot
]
table[row sep=crcr]{
10 2\\
12 4\\
14 5\\
17 9\\
21 13\\
25 18\\
30 24\\
36 31\\
43 40\\
51 50\\
62 64\\
74 80\\
89 100\\
106 123\\
127 151\\
153 186\\
183 227\\
220 278\\
264 338\\
316 410\\
};
\addplot [
color=red,
dashed,
forget plot
]
table[row sep=crcr]{
10 6\\
12 7\\
14 8\\
17 9\\
21 12\\
25 14\\
30 16\\
36 19\\
43 23\\
51 27\\
62 32\\
74 38\\
89 45\\
106 54\\
127 64\\
153 77\\
183 92\\
220 110\\
264 132\\
316 157\\
};
\addplot [
color=red,
solid,
forget plot
]
table[row sep=crcr]{
10 6\\
12 7\\
14 8\\
17 10\\
21 13\\
25 16\\
30 19\\
36 23\\
43 28\\
51 33\\
62 40\\
74 46\\
89 55\\
106 63\\
127 75\\
153 88\\
183 104\\
220 122\\
264 145\\
316 171\\
};
\addplot [
color=blue,
dashed,
forget plot
]
table[row sep=crcr]{
10 11\\
12 13\\
14 15\\
17 18\\
21 22\\
25 26\\
30 31\\
36 37\\
43 44\\
51 52\\
62 63\\
74 75\\
89 90\\
106 107\\
127 128\\
153 154\\
183 184\\
220 221\\
264 265\\
316 317\\
};
\addplot [
color=blue,
solid,
forget plot
]
table[row sep=crcr]{
10 18\\
12 20\\
14 23\\
17 27\\
21 31\\
25 36\\
30 41\\
36 48\\
43 55\\
51 64\\
62 75\\
74 88\\
89 103\\
106 121\\
127 143\\
153 169\\
183 200\\
220 237\\
264 282\\
316 334\\
};
\addplot [
color=black,
dashed,
forget plot
]
table[row sep=crcr]{
10 15\\
12 18\\
14 21\\
17 25\\
21 30\\
25 36\\
30 43\\
36 51\\
43 61\\
51 72\\
62 88\\
74 105\\
89 126\\
106 149\\
127 179\\
153 215\\
183 257\\
220 309\\
264 370\\
316 443\\
};
\addplot [
color=black,
solid,
forget plot
]
table[row sep=crcr]{
10 21\\
12 25\\
14 29\\
17 34\\
21 40\\
25 46\\
30 54\\
36 63\\
43 73\\
51 85\\
62 101\\
74 118\\
89 140\\
106 164\\
127 194\\
153 231\\
183 273\\
220 326\\
264 388\\
316 461\\
};
\addplot [
color=red,
solid,
forget plot
]
table[row sep=crcr]{
10 1\\
12 1\\
14 1\\
17 1\\
21 1\\
25 1\\
30 2\\
36 4\\
43 7\\
51 10\\
62 14\\
74 19\\
89 26\\
106 33\\
127 42\\
153 54\\
183 68\\
220 85\\
264 106\\
316 131\\
};
\addplot [
color=blue,
solid,
forget plot
]
table[row sep=crcr]{
10 1\\
12 1\\
14 1\\
17 3\\
21 6\\
25 9\\
30 12\\
36 17\\
43 23\\
51 30\\
62 40\\
74 51\\
89 65\\
106 81\\
127 101\\
153 125\\
183 154\\
220 190\\
264 233\\
316 284\\
};
\addplot [
color=black,
solid,
forget plot
]
table[row sep=crcr]{
10 2\\
12 4\\
14 5\\
17 9\\
21 13\\
25 18\\
30 24\\
36 31\\
43 40\\
51 50\\
62 64\\
74 80\\
89 100\\
106 123\\
127 151\\
153 186\\
183 227\\
220 278\\
264 338\\
316 410\\
};
\addplot [
color=red,
dashed,
forget plot
]
table[row sep=crcr]{
10 6\\
12 7\\
14 8\\
17 9\\
21 12\\
25 14\\
30 16\\
36 19\\
43 23\\
51 27\\
62 32\\
74 38\\
89 45\\
106 54\\
127 64\\
153 77\\
183 92\\
220 110\\
264 132\\
316 157\\
};
\addplot [
color=red,
solid,
forget plot
]
table[row sep=crcr]{
10 6\\
12 7\\
14 8\\
17 10\\
21 13\\
25 16\\
30 19\\
36 23\\
43 28\\
51 33\\
62 40\\
74 46\\
89 55\\
106 63\\
127 75\\
153 88\\
183 104\\
220 122\\
264 145\\
316 171\\
};
\addplot [
color=blue,
dashed,
forget plot
]
table[row sep=crcr]{
10 11\\
12 13\\
14 15\\
17 18\\
21 22\\
25 26\\
30 31\\
36 37\\
43 44\\
51 52\\
62 63\\
74 75\\
89 90\\
106 107\\
127 128\\
153 154\\
183 184\\
220 221\\
264 265\\
316 317\\
};
\addplot [
color=blue,
solid,
forget plot
]
table[row sep=crcr]{
10 18\\
12 20\\
14 23\\
17 27\\
21 31\\
25 36\\
30 41\\
36 48\\
43 55\\
51 64\\
62 75\\
74 88\\
89 103\\
106 121\\
127 143\\
153 169\\
183 200\\
220 237\\
264 282\\
316 334\\
};
\addplot [
color=black,
dashed,
forget plot
]
table[row sep=crcr]{
10 15\\
12 18\\
14 21\\
17 25\\
21 30\\
25 36\\
30 43\\
36 51\\
43 61\\
51 72\\
62 88\\
74 105\\
89 126\\
106 149\\
127 179\\
153 215\\
183 257\\
220 309\\
264 370\\
316 443\\
};
\addplot [
color=black,
solid,
forget plot
]
table[row sep=crcr]{
10 21\\
12 25\\
14 29\\
17 34\\
21 40\\
25 46\\
30 54\\
36 63\\
43 73\\
51 85\\
62 101\\
74 118\\
89 140\\
106 164\\
127 194\\
153 231\\
183 273\\
220 326\\
264 388\\
316 461\\
};
\addplot [
color=red,
solid,
forget plot
]
table[row sep=crcr]{
10 1\\
12 1\\
14 1\\
17 1\\
21 1\\
25 1\\
30 2\\
36 4\\
43 7\\
51 10\\
62 14\\
74 19\\
89 26\\
106 33\\
127 42\\
153 54\\
183 68\\
220 85\\
264 106\\
316 131\\
};
\addplot [
color=blue,
solid,
forget plot
]
table[row sep=crcr]{
10 1\\
12 1\\
14 1\\
17 3\\
21 6\\
25 9\\
30 12\\
36 17\\
43 23\\
51 30\\
62 40\\
74 51\\
89 65\\
106 81\\
127 101\\
153 125\\
183 154\\
220 190\\
264 233\\
316 284\\
};
\addplot [
color=black,
solid,
forget plot
]
table[row sep=crcr]{
10 2\\
12 4\\
14 5\\
17 9\\
21 13\\
25 18\\
30 24\\
36 31\\
43 40\\
51 50\\
62 64\\
74 80\\
89 100\\
106 123\\
127 151\\
153 186\\
183 227\\
220 278\\
264 338\\
316 410\\
};
\addplot [
color=red,
dashed,
forget plot
]
table[row sep=crcr]{
10 6\\
12 7\\
14 8\\
17 9\\
21 12\\
25 14\\
30 16\\
36 19\\
43 23\\
51 27\\
62 32\\
74 38\\
89 45\\
106 54\\
127 64\\
153 77\\
183 92\\
220 110\\
264 132\\
316 157\\
};
\addplot [
color=red,
solid,
forget plot
]
table[row sep=crcr]{
10 6\\
12 7\\
14 8\\
17 10\\
21 13\\
25 16\\
30 19\\
36 23\\
43 28\\
51 33\\
62 40\\
74 46\\
89 55\\
106 63\\
127 75\\
153 88\\
183 104\\
220 122\\
264 145\\
316 171\\
};
\addplot [
color=blue,
dashed,
forget plot
]
table[row sep=crcr]{
10 11\\
12 13\\
14 15\\
17 18\\
21 22\\
25 26\\
30 31\\
36 37\\
43 44\\
51 52\\
62 63\\
74 75\\
89 90\\
106 107\\
127 128\\
153 154\\
183 184\\
220 221\\
264 265\\
316 317\\
};
\addplot [
color=blue,
solid,
forget plot
]
table[row sep=crcr]{
10 18\\
12 20\\
14 23\\
17 27\\
21 31\\
25 36\\
30 41\\
36 48\\
43 55\\
51 64\\
62 75\\
74 88\\
89 103\\
106 121\\
127 143\\
153 169\\
183 200\\
220 237\\
264 282\\
316 334\\
};
\addplot [
color=black,
dashed,
forget plot
]
table[row sep=crcr]{
10 15\\
12 18\\
14 21\\
17 25\\
21 30\\
25 36\\
30 43\\
36 51\\
43 61\\
51 72\\
62 88\\
74 105\\
89 126\\
106 149\\
127 179\\
153 215\\
183 257\\
220 309\\
264 370\\
316 443\\
};
\addplot [
color=black,
solid,
forget plot
]
table[row sep=crcr]{
10 21\\
12 25\\
14 29\\
17 34\\
21 40\\
25 46\\
30 54\\
36 63\\
43 73\\
51 85\\
62 101\\
74 118\\
89 140\\
106 164\\
127 194\\
153 231\\
183 273\\
220 326\\
264 388\\
316 461\\
};
\addplot [
color=red,
solid,
forget plot
]
table[row sep=crcr]{
10 1\\
12 1\\
14 1\\
17 1\\
21 1\\
25 1\\
30 2\\
36 4\\
43 7\\
51 10\\
62 14\\
74 19\\
89 26\\
106 33\\
127 42\\
153 54\\
183 68\\
220 85\\
264 106\\
316 131\\
};
\addplot [
color=blue,
solid,
forget plot
]
table[row sep=crcr]{
10 1\\
12 1\\
14 1\\
17 3\\
21 6\\
25 9\\
30 12\\
36 17\\
43 23\\
51 30\\
62 40\\
74 51\\
89 65\\
106 81\\
127 101\\
153 125\\
183 154\\
220 190\\
264 233\\
316 284\\
};
\addplot [
color=black,
solid,
forget plot
]
table[row sep=crcr]{
10 2\\
12 4\\
14 5\\
17 9\\
21 13\\
25 18\\
30 24\\
36 31\\
43 40\\
51 50\\
62 64\\
74 80\\
89 100\\
106 123\\
127 151\\
153 186\\
183 227\\
220 278\\
264 338\\
316 410\\
};
\addplot [
color=red,
dashed,
forget plot
]
table[row sep=crcr]{
10 6\\
12 7\\
14 8\\
17 9\\
21 12\\
25 14\\
30 16\\
36 19\\
43 23\\
51 27\\
62 32\\
74 38\\
89 45\\
106 54\\
127 64\\
153 77\\
183 92\\
220 110\\
264 132\\
316 157\\
};
\addplot [
color=red,
solid,
forget plot
]
table[row sep=crcr]{
10 6\\
12 7\\
14 8\\
17 10\\
21 13\\
25 16\\
30 19\\
36 23\\
43 28\\
51 33\\
62 40\\
74 46\\
89 55\\
106 63\\
127 75\\
153 88\\
183 104\\
220 122\\
264 145\\
316 171\\
};
\addplot [
color=blue,
dashed,
forget plot
]
table[row sep=crcr]{
10 11\\
12 13\\
14 15\\
17 18\\
21 22\\
25 26\\
30 31\\
36 37\\
43 44\\
51 52\\
62 63\\
74 75\\
89 90\\
106 107\\
127 128\\
153 154\\
183 184\\
220 221\\
264 265\\
316 317\\
};
\addplot [
color=blue,
solid,
forget plot
]
table[row sep=crcr]{
10 18\\
12 20\\
14 23\\
17 27\\
21 31\\
25 36\\
30 41\\
36 48\\
43 55\\
51 64\\
62 75\\
74 88\\
89 103\\
106 121\\
127 143\\
153 169\\
183 200\\
220 237\\
264 282\\
316 334\\
};
\addplot [
color=black,
dashed,
forget plot
]
table[row sep=crcr]{
10 15\\
12 18\\
14 21\\
17 25\\
21 30\\
25 36\\
30 43\\
36 51\\
43 61\\
51 72\\
62 88\\
74 105\\
89 126\\
106 149\\
127 179\\
153 215\\
183 257\\
220 309\\
264 370\\
316 443\\
};
\addplot [
color=black,
solid,
forget plot
]
table[row sep=crcr]{
10 21\\
12 25\\
14 29\\
17 34\\
21 40\\
25 46\\
30 54\\
36 63\\
43 73\\
51 85\\
62 101\\
74 118\\
89 140\\
106 164\\
127 194\\
153 231\\
183 273\\
220 326\\
264 388\\
316 461\\
};
\addplot [
color=red,
solid,
forget plot
]
table[row sep=crcr]{
10 1\\
12 1\\
14 1\\
17 1\\
21 1\\
25 1\\
30 2\\
36 4\\
43 7\\
51 10\\
62 14\\
74 19\\
89 26\\
106 33\\
127 42\\
153 54\\
183 68\\
220 85\\
264 106\\
316 131\\
};
\addplot [
color=blue,
solid,
forget plot
]
table[row sep=crcr]{
10 1\\
12 1\\
14 1\\
17 3\\
21 6\\
25 9\\
30 12\\
36 17\\
43 23\\
51 30\\
62 40\\
74 51\\
89 65\\
106 81\\
127 101\\
153 125\\
183 154\\
220 190\\
264 233\\
316 284\\
};
\addplot [
color=black,
solid,
forget plot
]
table[row sep=crcr]{
10 2\\
12 4\\
14 5\\
17 9\\
21 13\\
25 18\\
30 24\\
36 31\\
43 40\\
51 50\\
62 64\\
74 80\\
89 100\\
106 123\\
127 151\\
153 186\\
183 227\\
220 278\\
264 338\\
316 410\\
};
\addplot [
color=red,
dashed,
forget plot
]
table[row sep=crcr]{
10 6\\
12 7\\
14 8\\
17 9\\
21 12\\
25 14\\
30 16\\
36 19\\
43 23\\
51 27\\
62 32\\
74 38\\
89 45\\
106 54\\
127 64\\
153 77\\
183 92\\
220 110\\
264 132\\
316 157\\
};
\addplot [
color=red,
solid,
forget plot
]
table[row sep=crcr]{
10 6\\
12 7\\
14 8\\
17 10\\
21 13\\
25 16\\
30 19\\
36 23\\
43 28\\
51 33\\
62 40\\
74 46\\
89 55\\
106 63\\
127 75\\
153 88\\
183 104\\
220 122\\
264 145\\
316 171\\
};
\addplot [
color=blue,
dashed,
forget plot
]
table[row sep=crcr]{
10 11\\
12 13\\
14 15\\
17 18\\
21 22\\
25 26\\
30 31\\
36 37\\
43 44\\
51 52\\
62 63\\
74 75\\
89 90\\
106 107\\
127 128\\
153 154\\
183 184\\
220 221\\
264 265\\
316 317\\
};
\addplot [
color=blue,
solid,
forget plot
]
table[row sep=crcr]{
10 18\\
12 20\\
14 23\\
17 27\\
21 31\\
25 36\\
30 41\\
36 48\\
43 55\\
51 64\\
62 75\\
74 88\\
89 103\\
106 121\\
127 143\\
153 169\\
183 200\\
220 237\\
264 282\\
316 334\\
};
\addplot [
color=black,
dashed,
forget plot
]
table[row sep=crcr]{
10 15\\
12 18\\
14 21\\
17 25\\
21 30\\
25 36\\
30 43\\
36 51\\
43 61\\
51 72\\
62 88\\
74 105\\
89 126\\
106 149\\
127 179\\
153 215\\
183 257\\
220 309\\
264 370\\
316 443\\
};
\addplot [
color=black,
solid,
forget plot
]
table[row sep=crcr]{
10 21\\
12 25\\
14 29\\
17 34\\
21 40\\
25 46\\
30 54\\
36 63\\
43 73\\
51 85\\
62 101\\
74 118\\
89 140\\
106 164\\
127 194\\
153 231\\
183 273\\
220 326\\
264 388\\
316 461\\
};
\end{axis}
\end{tikzpicture}%

%% file: results.tex
  
\section{Numerical Results}
\label{sec:results}
In this section we apply the algorithms from the previous section to a number of 
test problems. All tests were performed in \begin{sc} matlab \end{sc},
single threaded. The required fast matrix-vector products $\DM x$ and $\DM'y$
were implemented using ffts, and the \begin{sc} lapack \end{sc} routine \emph{dstevx}
was used for the tridiagonal eigenvalue problem.

To show the validity of our algorithms for different values of $T$, experiments are carried out for 
\begin{equation*}
T_1=1.1,\quad T_2=2,\quad T_3=3.8.
\end{equation*}
Following \cite{Adcock2013}, the product $L=2T\sm$ is held constant when varying $T$ in order to maintain a fixed condition number. Thus, in our computations we use
\begin{equation*}
M_1\approx\frac{2}{1.1}M_2,\quad M_3\approx\frac{2}{3.8}M_2,
\end{equation*}
with the restriction that $\bM$ is odd.
\subsection{Computational complexity}

Figure \ref{fig:timings} shows execution time for increasing degrees
of freedom of the algorithm for different values of $T$. The figure confirms 
the theorized ${\mathcal O}(N\log^2{N})$ asymptotic complexity of our algorithms. It also shows execution speed is on par
with the current fast algorithm by Lyon. Also, as in Lyons algorithm, the
majority of the work is in computing a low-rank matrix decomposition related to
$\DM$, in this case the middle part of the SVD.
This can provide a significant speedup when multiple approximations are needed with the same parameters.

Figure~\ref{fig:timings2} compares our two approaches in more
detail. The figure shows the execution time per degree of freedom. The
difference between algorithms is very much dependent on implementation, but it highlights a trend seen for
different values of $T$.  Both algorithms appear to be slightly faster for
larger $T$, which is somewhat counterintuitive. After all, the cost of the FFT
operations is ${\mathcal O}(L)$, which is held constant for all experiments. However, both approaches solve a subproblem
with a cost that grows as ${\mathcal O}(\bM)$.  This is either the tridiagonal eigenvalue
problem with a matrix of size $\bM\times\bM$, or a linear solve of an
$\bM\times\log{\bM}$ system. Due to our scaling of $\bM$, these costs
actually decrease with $T$.

\setlength{\figurewidth}{5.75cm}
\setlength{\figureheight}{5.75cm}
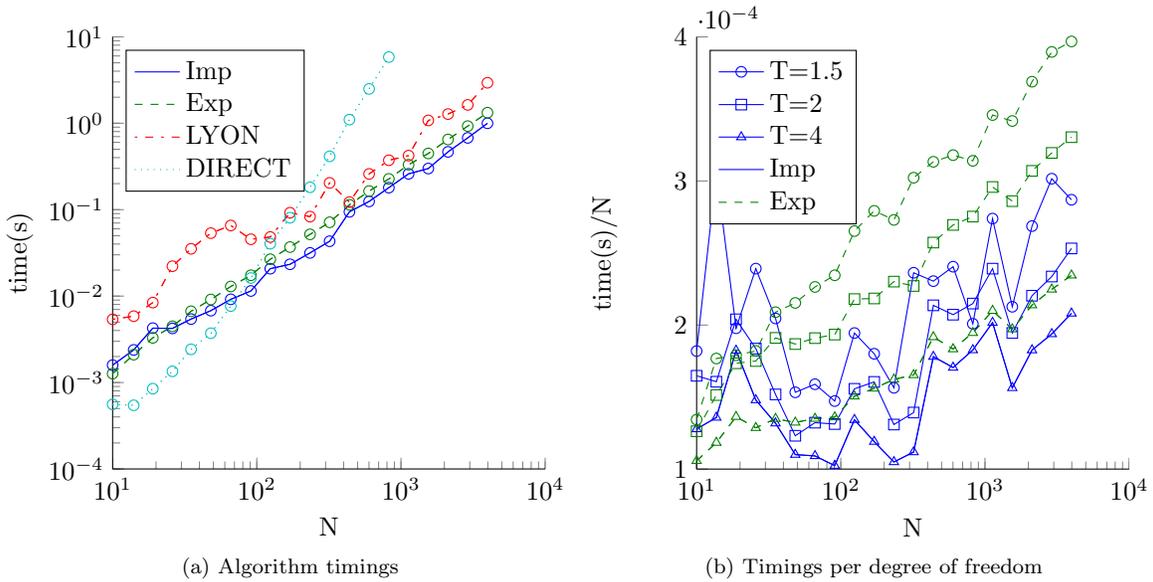
\begin{figure}[htbp]
\begin{center}
\subfloat[][Algorithm timings]{\input{img/PAvsTRIDtime1timings.tex} \label{fig:timings}}
\subfloat[][Timings per degree of freedom]{\input{img/PAvsTRIDtime21timings.tex} \label{fig:timings2}}
\end{center}
\caption{Execution time for increasing degrees of freedom $N$, for MATLAB
  implementations of the Explicit and Implicit projection algorithms, the Lyon algorithm and a direct solver.}
\label{fig:lyonvsexplicitspeed}
\end{figure}

\subsection{Convergence}
The accuracy of the solution obtained through our algorithms is shown in Figures
\ref{fig:accuracyf1} through \ref{fig:accuracyf4}. The accuracy is measured as the
maximum pointwise error over an equi-sampled grid ten times denser than the one
used for construction. This is measured for increasing number of degrees of
freedom $N$ for four test functions:
\begin{itemize}
\item A well-behaved, smooth function to show convergence in near optimal conditions,
  \begin{equation*}
    f_1(x)=x^2.
  \end{equation*}
\item A highly oscillatory function, to show the resolution power of Fourier
  extensions for oscillatory functions,
  \begin{equation*}
    f_2(x)=\mbox{Ai}(76x).
  \end{equation*}
\item A function with a pole in the complex plane near the real interval   $[-1,1]$, to
  show convergence at a slower, but still exponential rate,
  \begin{equation*}
    f_3(x)=\frac{1}{1.1-x^2}.
  \end{equation*}
\item A function with discontinuous first derivative, to see the breakdown of
  the algorithm to algebraic convergence speeds,
  \begin{equation*}
    f_4(x)=|x|.
  \end{equation*}
\end{itemize}
 The convergence behaviour seen is in accordance with \cite{Huybrechs2010},
 \cite{Adcock2012} and \cite{Adcock2013}. Convergence for functions analytic in
 $[-1,1]$ is at least geometric, even when singularities are present near the
 real interval. Following the earlier arguments about resolution power from \S\ref{sec:extensionlength}, Fourier extensions of oscillatory functions start to converge sooner for lower
values of $T$. Note that this is in terms of degrees of freedom, and that we
increased the oversampling for lower $T$ to maintain conditioning. 

When the approximant is in $C^k$, having continuous derivatives up to $f^{(k)}$,
the convergence rate becomes algebraic at a rate of ${\mathcal O}(N^{-k+1})$. The
convergence of order ${\mathcal O}(N^{-1})$ for the $C^0$ function is apparent from Figure~\ref{fig:accuracyf4}.

\setlength{\figureheight}{4cm}
\setlength{\figurewidth}{5.75cm}
\begin{figure}[htbp]
\begin{center}
\subfloat[][$f_1(x)=x^2$]{\input{img/PAvsTRIDf1res.tex} \label{fig:accuracyf1}}
\subfloat[][$f_2(x)=\mbox{Ai}(76x)$]{\input{img/PAvsTRIDf2res.tex} \label{fig:accuracyf2}}\\
\subfloat[][$f_3(x)=\frac{1}{1.1-x^2}$]{\input{img/PAvsTRIDf3res.tex} \label{fig:accuracyf3}}
\subfloat[][$f_4(x)=|x|$]{\input{img/PAvsTRIDf4res.tex} \label{fig:accuracyf4}}
\caption{The $L_\infty$ norm of the error, computed by oversampling the solution
  by a factor 10, for the various testfunctions}
\end{center}
\end{figure}
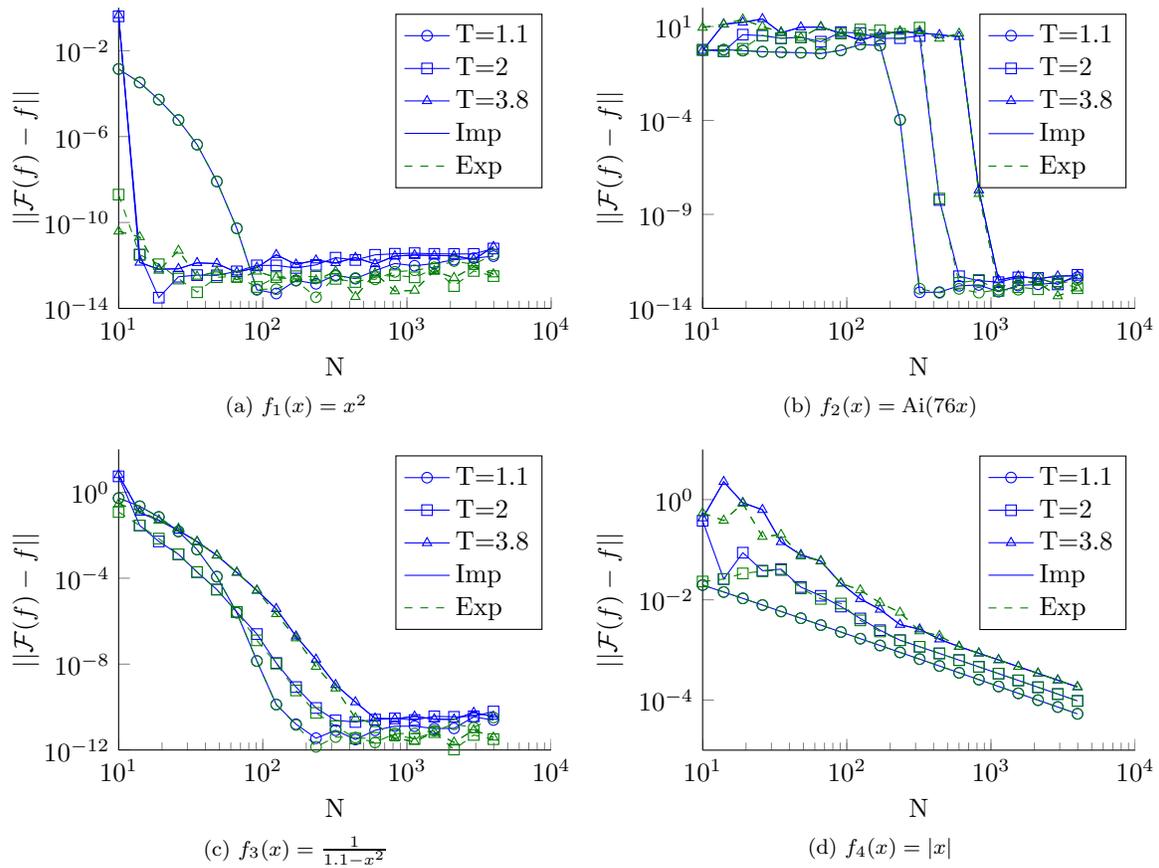


\subsection{Robustness}
To ensure there is no accumulation of error for very large $N$, Figure \ref{fig:robustnessaccuracy} shows the continuation of the accuracy experiment for $N$ up to $10^5$. No error accumulation is visible, the error fluctuates around the cutoff threshold. 

Figure \ref{fig:robust1} shows the challenging problem of approximating an oscillatory function with a frequency that increases with $N$, $f(x)=\sin(Nx)$. The error for the Algorithm using 
$T=1.1$ and $T=2$ stays close to machine precision, if maybe slightly increasing. However, for $T=3.8$, the maximum frequency in the Fourier basis is lower than the frequency of the signal for every $N$, so there cannot be any convergence. This experiment shows that even a very oscillatory signal such as $f(x)=\sin(10^5x)$ can be accurately approximated, as long as there are sufficient degrees of freedom for the chosen value of $T$.
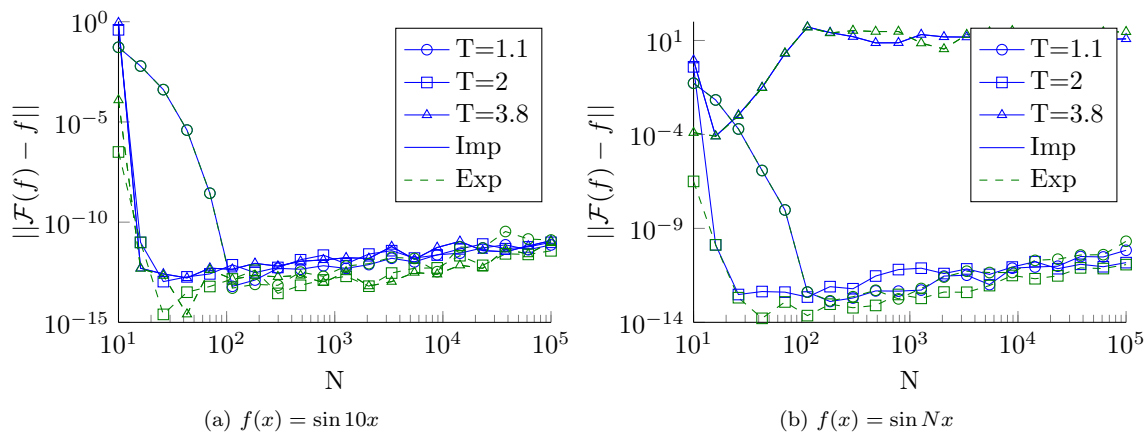
\begin{figure}[htbp]
\begin{center}
\subfloat[][$f(x)=\sin{10x}$]{\input{img/robustnessaccuracyres.tex}\label{fig:robust1}}
\subfloat[][$f(x)=\sin{Nx}$]{\input{img/robustnessaccuracy2res.tex}\label{fig:robust2}}
\caption{Illustration of the robustness of FE approximations for large $N$.}
\label{fig:robustnessaccuracy}
\end{center}
\end{figure}

%% file: img/PAvsTRIDtime1timings.tex
%
%
\definecolor{mycolor1}{rgb}{0,0.75,0.75}%
\begin{tikzpicture}

\begin{axis}[%
width=\figurewidth,
height=\figureheight,
unbounded coords=jump,
scale only axis,
xmode=log,
xmin=10,
xmax=10000,
xminorticks=true,
xlabel={N},
ymode=log,
ymin=0.0001,
ymax=10,
yminorticks=true,
ylabel={time(s)},
axis x line*=bottom,
axis y line*=left,
legend style={at={(0.03,0.97)},anchor=north west,draw=black,fill=white,legend cell align=left}
]
\addplot [
color=blue,
solid,
mark=o,
mark options={solid},
forget plot
]
table[row sep=crcr]{
10 0.00158921033333333\\
14 0.00238267688888889\\
19 0.00425643955555556\\
26 0.00422985133333333\\
35 0.00542219611111111\\
48 0.00678248844444444\\
66 0.00917212322222222\\
91 0.0114609794444444\\
124 0.0207485328888889\\
170 0.0233258025555556\\
234 0.0315919745555555\\
320 0.0430963511111111\\
439 0.0943696772222222\\
601 0.124470477222222\\
824 0.179703761\\
1129 0.259623510666667\\
1547 0.298943706333333\\
2120 0.466132372222222\\
2905 0.676942902111111\\
3981 0.998114841777778\\
};
\addplot [
color=blue,
solid
]
table[row sep=crcr]{
10 0.00158921033333333\\
14 0.00238267688888889\\
19 0.00425643955555556\\
26 0.00422985133333333\\
35 0.00542219611111111\\
48 0.00678248844444444\\
66 0.00917212322222222\\
91 0.0114609794444444\\
124 0.0207485328888889\\
170 0.0233258025555556\\
234 0.0315919745555555\\
320 0.0430963511111111\\
439 0.0943696772222222\\
601 0.124470477222222\\
824 0.179703761\\
1129 0.259623510666667\\
1547 0.298943706333333\\
2120 0.466132372222222\\
2905 0.676942902111111\\
3981 0.998114841777778\\
};
\addlegendentry{\PA};

\addplot [
color=green!50!black,
dashed,
mark=o,
mark options={solid},
forget plot
]
table[row sep=crcr]{
10 0.00126661944444444\\
14 0.002100967\\
19 0.003289514\\
26 0.00447091033333333\\
35 0.00667287788888889\\
48 0.009128971\\
66 0.0129163621111111\\
91 0.0174196573333333\\
124 0.0266822212222222\\
170 0.0369045878888889\\
234 0.0517871566666667\\
320 0.0714504551111111\\
439 0.113912137777778\\
601 0.164097074666667\\
824 0.226815048222222\\
1129 0.331138090555556\\
1547 0.446875352111111\\
2120 0.648517979222222\\
2905 0.925644533222222\\
3981 1.32352836255556\\
};
\addplot [
color=green!50!black,
dashed
]
table[row sep=crcr]{
10 0.00126661944444444\\
14 0.002100967\\
19 0.003289514\\
26 0.00447091033333333\\
35 0.00667287788888889\\
48 0.009128971\\
66 0.0129163621111111\\
91 0.0174196573333333\\
124 0.0266822212222222\\
170 0.0369045878888889\\
234 0.0517871566666667\\
320 0.0714504551111111\\
439 0.113912137777778\\
601 0.164097074666667\\
824 0.226815048222222\\
1129 0.331138090555556\\
1547 0.446875352111111\\
2120 0.648517979222222\\
2905 0.925644533222222\\
3981 1.32352836255556\\
};
\addlegendentry{\TRID};

\addplot [
color=red,
dash pattern=on 1pt off 3pt on 3pt off 3pt,
mark=o,
mark options={solid},
forget plot
]
table[row sep=crcr]{
10 0.00538251677777778\\
14 0.00585620355555556\\
19 0.00841398922222222\\
26 0.0221634881111111\\
35 0.0352249572222222\\
48 0.0534908131111111\\
66 0.0655445123333333\\
91 0.0454206701111111\\
124 0.0482299484444444\\
170 0.0921030048888889\\
234 0.0830119571111111\\
320 0.203923882111111\\
439 0.122364598777778\\
601 0.258999347111111\\
824 0.372781147777778\\
1129 0.421550786777778\\
1547 1.07728683344444\\
2120 1.27526472544444\\
2905 1.62870139933333\\
3981 2.937998262\\
};
\addplot [
color=red,
dash pattern=on 1pt off 3pt on 3pt off 3pt
]
table[row sep=crcr]{
10 0.00538251677777778\\
14 0.00585620355555556\\
19 0.00841398922222222\\
26 0.0221634881111111\\
35 0.0352249572222222\\
48 0.0534908131111111\\
66 0.0655445123333333\\
91 0.0454206701111111\\
124 0.0482299484444444\\
170 0.0921030048888889\\
234 0.0830119571111111\\
320 0.203923882111111\\
439 0.122364598777778\\
601 0.258999347111111\\
824 0.372781147777778\\
1129 0.421550786777778\\
1547 1.07728683344444\\
2120 1.27526472544444\\
2905 1.62870139933333\\
3981 2.937998262\\
};
\addlegendentry{LYON};

\addplot [
color=mycolor1,
dotted,
mark=o,
mark options={solid},
forget plot
]
table[row sep=crcr]{
10 0.000559158777777778\\
14 0.000545093555555556\\
19 0.000848936666666667\\
26 0.00134784277777778\\
35 0.00242280988888889\\
48 0.00371684088888889\\
66 0.007660124\\
91 0.01617018\\
124 0.0404058953333333\\
170 0.0806618271111111\\
234 0.181522757444444\\
320 0.414113659111111\\
439 1.10100058344444\\
601 2.49419437244444\\
824 5.821768966\\
1129 NaN\\
1547 NaN\\
2120 NaN\\
2905 NaN\\
3981 NaN\\
};
\addplot [
color=mycolor1,
dotted
]
table[row sep=crcr]{
10 0.000559158777777778\\
14 0.000545093555555556\\
19 0.000848936666666667\\
26 0.00134784277777778\\
35 0.00242280988888889\\
48 0.00371684088888889\\
66 0.007660124\\
91 0.01617018\\
124 0.0404058953333333\\
170 0.0806618271111111\\
234 0.181522757444444\\
320 0.414113659111111\\
439 1.10100058344444\\
601 2.49419437244444\\
824 5.821768966\\
1129 NaN\\
1547 NaN\\
2120 NaN\\
2905 NaN\\
3981 NaN\\
};
\addlegendentry{DIRECT};

\end{axis}
\end{tikzpicture}%

%% file: img/PAvsTRIDtime21timings.tex
%
\begin{tikzpicture}

\begin{axis}[%
width=\figurewidth,
height=\figureheight,
scale only axis,
xmode=log,
xmin=10,
xmax=10000,
xminorticks=true,
xlabel={N},
ymin=0.0001,
ymax=0.0004,
ylabel={time(s)/N},
axis x line*=bottom,
axis y line*=left,
legend style={at={(0.03,0.97)},anchor=north west,draw=black,fill=white,legend cell align=left}
]
\addplot [
color=blue,
solid,
mark=o,
mark options={solid}
]
table[row sep=crcr]{
10 0.000181925622222222\\
13.7038345066317 0.000306036013997349\\
18.7795080185149 0.000197762990542942\\
25.7351270001691 0.000239144807793626\\
35.2669921417466 0.000204580153523389\\
48.3293023857175 0.00015330130240385\\
66.2296761714833 0.000158902773357566\\
90.7600521681814 0.000147225259384621\\
124.376073472602 0.000194427019177038\\
170.442912745319 0.000180047498701021\\
233.572146909012 0.00015641967463036\\
320.083404659976 0.000236271855012651\\
438.637000577954 0.00023046697631304\\
601.100886440559 0.000240496494265583\\
823.738706957101 0.000200857566769592\\
1128.83789168469 0.000273875412393994\\
1546.9407652462 0.000212686125094614\\
2119.90202384961 0.000268668021935362\\
2905.07865051086 0.000301568058200936\\
3981.07170553497 0.000286953255764033\\
};
\addlegendentry{T=1.5};

\addplot [
color=blue,
solid,
mark=square,
mark options={solid}
]
table[row sep=crcr]{
10 0.000164767911111111\\
13.7038345066317 0.000160643204485836\\
18.7795080185149 0.000203967483931078\\
25.7351270001691 0.000183649651741764\\
35.2669921417466 0.000151805070324806\\
48.3293023857175 0.000123119755410485\\
66.2296761714833 0.00013222048811529\\
90.7600521681814 0.000131276781699918\\
124.376073472602 0.000155703453006797\\
170.442912745319 0.000160512808159029\\
233.572146909012 0.000130931745949628\\
320.083404659976 0.000139243131689419\\
438.637000577954 0.000213667444553264\\
601.100886440559 0.000207176831392537\\
823.738706957101 0.00021488695782144\\
1128.83789168469 0.000239122211917168\\
1546.9407652462 0.000194495403073034\\
2119.90202384961 0.00022034465763584\\
2905.07865051086 0.000233656445905927\\
3981.07170553497 0.000253162993003136\\
};
\addlegendentry{T=2};

\addplot [
color=blue,
solid,
mark=triangle,
mark options={solid}
]
table[row sep=crcr]{
10 0.000127818788888889\\
13.7038345066317 0.00013581958467248\\
18.7795080185149 0.000182292315465606\\
25.7351270001691 0.000147936532212156\\
35.2669921417466 0.000131829706472721\\
48.3293023857175 0.000110042487769042\\
66.2296761714833 0.000109129882487479\\
90.7600521681814 0.000102409208556731\\
124.376073472602 0.000134237949831959\\
170.442912745319 0.000119070467680039\\
233.572146909012 0.000104932336476997\\
320.083404659976 0.000111830432426138\\
438.637000577954 0.000178039377301837\\
601.100886440559 0.000170506769386302\\
823.738706957101 0.000182666718161222\\
1128.83789168469 0.00020166396059878\\
1546.9407652462 0.000156386422581423\\
2119.90202384961 0.000182603099252538\\
2905.07865051086 0.00019378828258073\\
3981.07170553497 0.000208084805032505\\
};
\addlegendentry{T=4};

\addplot [
color=blue,
solid
]
table[row sep=crcr]{
10 0.000127818788888889\\
13.7038345066317 0.00013581958467248\\
18.7795080185149 0.000182292315465606\\
25.7351270001691 0.000147936532212156\\
35.2669921417466 0.000131829706472721\\
48.3293023857175 0.000110042487769042\\
66.2296761714833 0.000109129882487479\\
90.7600521681814 0.000102409208556731\\
124.376073472602 0.000134237949831959\\
170.442912745319 0.000119070467680039\\
233.572146909012 0.000104932336476997\\
320.083404659976 0.000111830432426138\\
438.637000577954 0.000178039377301837\\
601.100886440559 0.000170506769386302\\
823.738706957101 0.000182666718161222\\
1128.83789168469 0.00020166396059878\\
1546.9407652462 0.000156386422581423\\
2119.90202384961 0.000182603099252538\\
2905.07865051086 0.00019378828258073\\
3981.07170553497 0.000208084805032505\\
};
\addlegendentry{\PA};

\addplot [
color=green!50!black,
dashed,
mark=o,
mark options={solid},
forget plot
]
table[row sep=crcr]{
10 0.000134326688888889\\
13.7038345066317 0.000176762131710491\\
18.7795080185149 0.00017900036435798\\
25.7351270001691 0.000182025688942256\\
35.2669921417466 0.00020892385634663\\
48.3293023857175 0.000215338910755908\\
66.2296761714833 0.000226429033364541\\
90.7600521681814 0.000234528207342698\\
124.376073472602 0.000265227801734892\\
170.442912745319 0.000279249941031299\\
233.572146909012 0.000272984259940484\\
320.083404659976 0.000302256288108882\\
438.637000577954 0.000313327211432536\\
601.100886440559 0.000317942781467977\\
823.738706957101 0.000313985580262842\\
1128.83789168469 0.000345773938635571\\
1546.9407652462 0.000341628313037501\\
2119.90202384961 0.000368904696475796\\
2905.07865051086 0.00038956839614533\\
3981.07170553497 0.000396857689580048\\
};
\addplot [
color=green!50!black,
dashed,
mark=square,
mark options={solid},
forget plot
]
table[row sep=crcr]{
10 0.000126261111111111\\
13.7038345066317 0.00015138620095277\\
18.7795080185149 0.000173441829200795\\
25.7351270001691 0.000175019545160691\\
35.2669921417466 0.000191072893178366\\
48.3293023857175 0.000186797245354476\\
66.2296761714833 0.00019090835706231\\
90.7600521681814 0.000193372578728929\\
124.376073472602 0.000218011325728496\\
170.442912745319 0.000218438019056\\
233.572146909012 0.000230139269316055\\
320.083404659976 0.000227111000332696\\
438.637000577954 0.000257244758868424\\
601.100886440559 0.000269479289100289\\
823.738706957101 0.000275274317607432\\
1128.83789168469 0.000295810224749008\\
1546.9407652462 0.000285994425359791\\
2119.90202384961 0.000307015101175027\\
2905.07865051086 0.000319469489235158\\
3981.07170553497 0.000330399818653555\\
};
\addplot [
color=green!50!black,
dashed,
mark=triangle,
mark options={solid},
forget plot
]
table[row sep=crcr]{
10 0.000105739422222222\\
13.7038345066317 0.000118363781019226\\
18.7795080185149 0.000136340502266152\\
25.7351270001691 0.000128669062586903\\
35.2669921417466 0.000134959208151444\\
48.3293023857175 0.000132391068857863\\
66.2296761714833 0.000135046547531848\\
90.7600521681814 0.000136066966743431\\
124.376073472602 0.000150536079180606\\
170.442912745319 0.000156017681062132\\
233.572146909012 0.00016225956757729\\
320.083404659976 0.000165351191132341\\
438.637000577954 0.000191681172967613\\
601.100886440559 0.00018330083673124\\
823.738706957101 0.000194610226629538\\
1128.83789168469 0.000209956026233571\\
1546.9407652462 0.000197146991422933\\
2119.90202384961 0.000213823071176751\\
2905.07865051086 0.000224735933319423\\
3981.07170553497 0.000234504428245574\\
};
\addplot [
color=green!50!black,
dashed
]
table[row sep=crcr]{
10 0.000105739422222222\\
13.7038345066317 0.000118363781019226\\
18.7795080185149 0.000136340502266152\\
25.7351270001691 0.000128669062586903\\
35.2669921417466 0.000134959208151444\\
48.3293023857175 0.000132391068857863\\
66.2296761714833 0.000135046547531848\\
90.7600521681814 0.000136066966743431\\
124.376073472602 0.000150536079180606\\
170.442912745319 0.000156017681062132\\
233.572146909012 0.00016225956757729\\
320.083404659976 0.000165351191132341\\
438.637000577954 0.000191681172967613\\
601.100886440559 0.00018330083673124\\
823.738706957101 0.000194610226629538\\
1128.83789168469 0.000209956026233571\\
1546.9407652462 0.000197146991422933\\
2119.90202384961 0.000213823071176751\\
2905.07865051086 0.000224735933319423\\
3981.07170553497 0.000234504428245574\\
};
\addlegendentry{\TRID};

\end{axis}
\end{tikzpicture}%

%% file: img/PAvsTRIDf1res.tex
%
\begin{tikzpicture}

\begin{axis}[%
width=\figurewidth,
height=\figureheight,
scale only axis,
xmode=log,
xmin=10,
xmax=10000,
xminorticks=true,
xlabel={N},
ymode=log,
ymin=1e-14,
ymax=1,
yminorticks=true,
ylabel={\resnorm},
axis x line*=bottom,
axis y line*=left,
legend style={draw=black,fill=white,legend cell align=left}
]
\addplot [
color=blue,
solid,
mark=o,
mark options={solid}
]
table[row sep=crcr]{
10 0.00140766632059471\\
14 0.000329396862810638\\
19 5.25141450899562e-05\\
26 5.88059036141169e-06\\
35 4.23310023522538e-07\\
48 8.13714845508564e-09\\
66 5.35060785855106e-11\\
91 7.33323589689033e-14\\
124 4.92270931951236e-14\\
170 2.01788359074169e-13\\
234 1.39184999664343e-13\\
320 3.54187674430475e-13\\
439 2.41611019996001e-13\\
601 5.7308385855121e-13\\
824 1.23111664197581e-12\\
1129 9.6470226530864e-13\\
1547 1.22699425820105e-12\\
2120 1.840615632148e-12\\
2905 2.17402564161438e-12\\
3981 2.69521683089178e-12\\
};
\addlegendentry{T=1.1};

\addplot [
color=blue,
solid,
mark=square,
mark options={solid}
]
table[row sep=crcr]{
10 0.396923272198611\\
14 3.16302629184718e-12\\
19 3.15629274218997e-14\\
26 2.96728824206225e-13\\
35 3.53501108848778e-13\\
48 3.45200409047392e-13\\
66 5.28780736676241e-13\\
91 1.00419639841351e-12\\
124 1.00189807093249e-12\\
170 7.63144848586714e-13\\
234 1.05949772152077e-12\\
320 2.29818247511655e-12\\
439 1.79263167891276e-12\\
601 3.11353633117069e-12\\
824 3.49291079134011e-12\\
1129 3.7209928344767e-12\\
1547 3.46164189674837e-12\\
2120 3.41704729705788e-12\\
2905 3.43518529638852e-12\\
3981 6.13107163233041e-12\\
};
\addlegendentry{T=2};

\addplot [
color=blue,
solid,
mark=triangle,
mark options={solid}
]
table[row sep=crcr]{
10 0.457008595215443\\
14 1.37433928156666e-12\\
19 6.64165525428504e-13\\
26 6.99340779982841e-13\\
35 1.31710078133248e-12\\
48 1.21148572191377e-12\\
66 4.78742096521801e-13\\
91 8.029883813088e-13\\
124 3.12428603653763e-12\\
170 1.11768693220807e-12\\
234 1.74800102960042e-12\\
320 1.31082771610787e-12\\
439 2.29865981117202e-12\\
601 1.1720620302799e-12\\
824 2.73248556569262e-12\\
1129 3.04728926633802e-12\\
1547 2.88662126589903e-12\\
2120 2.8200903362884e-12\\
2905 1.90677400208752e-12\\
3981 7.3724891559427e-12\\
};
\addlegendentry{T=3.8};

\addplot [
color=blue,
solid
]
table[row sep=crcr]{
10 0.457008595215443\\
14 1.37433928156666e-12\\
19 6.64165525428504e-13\\
26 6.99340779982841e-13\\
35 1.31710078133248e-12\\
48 1.21148572191377e-12\\
66 4.78742096521801e-13\\
91 8.029883813088e-13\\
124 3.12428603653763e-12\\
170 1.11768693220807e-12\\
234 1.74800102960042e-12\\
320 1.31082771610787e-12\\
439 2.29865981117202e-12\\
601 1.1720620302799e-12\\
824 2.73248556569262e-12\\
1129 3.04728926633802e-12\\
1547 2.88662126589903e-12\\
2120 2.8200903362884e-12\\
2905 1.90677400208752e-12\\
3981 7.3724891559427e-12\\
};
\addlegendentry{\PA};

\addplot [
color=green!50!black,
dashed,
mark=o,
mark options={solid},
forget plot
]
table[row sep=crcr]{
10 0.00140025077904249\\
14 0.000328746577774597\\
19 5.22337228000547e-05\\
26 5.84926431501653e-06\\
35 4.23185428810591e-07\\
48 8.13712408564982e-09\\
66 5.34940980437819e-11\\
91 7.76046334161932e-14\\
124 7.80489771246425e-14\\
170 1.94626205527819e-13\\
234 3.28628516449918e-14\\
320 3.23412087935581e-13\\
439 2.58580654054796e-13\\
601 2.25153714494088e-13\\
824 5.02166688606455e-13\\
1129 5.9243956119091e-13\\
1547 8.17346773048188e-13\\
2120 1.60205490382644e-12\\
2905 1.11384221114665e-12\\
3981 3.2730489430243e-12\\
};
\addplot [
color=green!50!black,
dashed,
mark=square,
mark options={solid},
forget plot
]
table[row sep=crcr]{
10 1.98395538886233e-09\\
14 3.16258130879448e-12\\
19 1.13209442186753e-12\\
26 2.44360087978274e-13\\
35 5.61772908575764e-14\\
48 2.82551759824028e-13\\
66 3.00759417765481e-13\\
91 9.93649898866142e-14\\
124 2.84772225170595e-13\\
170 3.1397107161861e-13\\
234 2.94875236836301e-13\\
320 2.36477551656206e-13\\
439 1.68864941390275e-13\\
601 2.80664406887882e-13\\
824 3.51274664761208e-13\\
1129 2.78332930785441e-13\\
1547 6.85340696352079e-13\\
2120 1.09134928468376e-13\\
2905 5.51337096929101e-13\\
3981 3.18412221175833e-13\\
};
\addplot [
color=green!50!black,
dashed,
mark=triangle,
mark options={solid},
forget plot
]
table[row sep=crcr]{
10 3.86308762763126e-11\\
14 2.12548867404332e-11\\
19 5.78537222740647e-13\\
26 4.9340531742801e-12\\
35 3.04423175969746e-13\\
48 5.10480547327146e-13\\
66 2.17714797496031e-13\\
91 5.27022894244452e-13\\
124 2.74336109391809e-13\\
170 2.31481507335914e-13\\
234 2.04392085260084e-13\\
320 5.29798555012396e-13\\
439 3.4530110713911e-14\\
601 4.30767295416904e-13\\
824 6.39510670230098e-14\\
1129 6.8060850313632e-14\\
1547 5.52225602115535e-13\\
2120 2.66857329587461e-13\\
2905 8.74752099259061e-13\\
3981 3.90137794929268e-13\\
};
\addplot [
color=green!50!black,
dashed
]
table[row sep=crcr]{
10 3.86308762763126e-11\\
14 2.12548867404332e-11\\
19 5.78537222740647e-13\\
26 4.9340531742801e-12\\
35 3.04423175969746e-13\\
48 5.10480547327146e-13\\
66 2.17714797496031e-13\\
91 5.27022894244452e-13\\
124 2.74336109391809e-13\\
170 2.31481507335914e-13\\
234 2.04392085260084e-13\\
320 5.29798555012396e-13\\
439 3.4530110713911e-14\\
601 4.30767295416904e-13\\
824 6.39510670230098e-14\\
1129 6.8060850313632e-14\\
1547 5.52225602115535e-13\\
2120 2.66857329587461e-13\\
2905 8.74752099259061e-13\\
3981 3.90137794929268e-13\\
};
\addlegendentry{\TRID};

\end{axis}
\end{tikzpicture}%

%% file: img/PAvsTRIDf2res.tex
%
\begin{tikzpicture}

\begin{axis}[%
width=\figurewidth,
height=\figureheight,
scale only axis,
xmode=log,
xmin=10,
xmax=10000,
xminorticks=true,
xlabel={N},
ymode=log,
ymin=1e-14,
ymax=100,
yminorticks=true,
ylabel={\resnorm},
axis x line*=bottom,
axis y line*=left,
legend style={draw=black,fill=white,legend cell align=left}
]
\addplot [
color=blue,
solid,
mark=o,
mark options={solid}
]
table[row sep=crcr]{
10 0.54970122713218\\
14 0.581655790721302\\
19 0.520638326749275\\
26 0.455458835326001\\
35 0.425202878027033\\
48 0.402852866848304\\
66 0.365048422384967\\
91 0.539663953091967\\
124 1.09621109398667\\
170 0.971682459437309\\
234 0.000106020575697735\\
320 7.11932868000538e-14\\
439 7.26376229661107e-14\\
601 1.57226432551035e-13\\
824 1.75214192131025e-13\\
1129 7.67261281429559e-14\\
1547 1.63679323007278e-13\\
2120 1.89705400483431e-13\\
2905 2.24046298593796e-13\\
3981 4.55144784529672e-13\\
};
\addlegendentry{T=1.1};

\addplot [
color=blue,
solid,
mark=square,
mark options={solid}
]
table[row sep=crcr]{
10 0.573766109644415\\
14 0.481276742665816\\
19 3.66953077963745\\
26 3.23971242175184\\
35 2.35295277249942\\
48 2.64919596792481\\
66 1.54914218569941\\
91 4.98059191921126\\
124 4.43821683843362\\
170 2.27839939762996\\
234 2.39256867525303\\
320 3.1189651906739\\
439 5.82235071780331e-09\\
601 5.10025010992483e-13\\
824 2.95776590010516e-13\\
1129 2.46898987931899e-13\\
1547 4.43988788718106e-13\\
2120 4.12934353855841e-13\\
2905 1.86728305723294e-13\\
3981 6.07668363781462e-13\\
};
\addlegendentry{T=2};

\addplot [
color=blue,
solid,
mark=triangle,
mark options={solid}
]
table[row sep=crcr]{
10 0.51949386747693\\
14 13.2919779901681\\
19 16.5987324867404\\
26 25.154633479787\\
35 4.58625722807862\\
48 9.17235195049557\\
66 9.26408163172894\\
91 4.3253715735936\\
124 1.81255625130645\\
170 3.7585082265731\\
234 5.35648085443182\\
320 4.51488019804559\\
439 3.58689708659182\\
601 2.83453167715959\\
824 1.96979803960127e-08\\
1129 3.17686382629492e-13\\
1547 5.17207436607386e-13\\
2120 3.77974049557386e-13\\
2905 5.17313582703118e-13\\
3981 5.4844592034921e-13\\
};
\addlegendentry{T=3.8};

\addplot [
color=blue,
solid
]
table[row sep=crcr]{
10 0.51949386747693\\
14 13.2919779901681\\
19 16.5987324867404\\
26 25.154633479787\\
35 4.58625722807862\\
48 9.17235195049557\\
66 9.26408163172894\\
91 4.3253715735936\\
124 1.81255625130645\\
170 3.7585082265731\\
234 5.35648085443182\\
320 4.51488019804559\\
439 3.58689708659182\\
601 2.83453167715959\\
824 1.96979803960127e-08\\
1129 3.17686382629492e-13\\
1547 5.17207436607386e-13\\
2120 3.77974049557386e-13\\
2905 5.17313582703118e-13\\
3981 5.4844592034921e-13\\
};
\addlegendentry{\PA};

\addplot [
color=green!50!black,
dashed,
mark=o,
mark options={solid},
forget plot
]
table[row sep=crcr]{
10 0.54929741142842\\
14 0.582265339843337\\
19 0.518648096877321\\
26 0.454232590469682\\
35 0.425377072784239\\
48 0.402861313601309\\
66 0.364810564372454\\
91 0.539555643039873\\
124 1.09530807809145\\
170 0.971053277741054\\
234 0.000109217479886174\\
320 1.10883535097138e-13\\
439 6.66134507978936e-14\\
601 1.0876028067233e-13\\
824 6.83483109100763e-14\\
1129 9.78679882887349e-14\\
1547 9.77551695009754e-14\\
2120 3.9782079751505e-13\\
2905 2.01464856444472e-13\\
3981 3.24948514195255e-13\\
};
\addplot [
color=green!50!black,
dashed,
mark=square,
mark options={solid},
forget plot
]
table[row sep=crcr]{
10 0.562655392727446\\
14 0.48108523672094\\
19 0.651372660821783\\
26 3.36395414253793\\
35 2.3963833451324\\
48 2.62929455741685\\
66 1.13873973674801\\
91 4.19091629251215\\
124 6.90591244458963\\
170 6.16846528008663\\
234 4.04644098073169\\
320 8.67787904990297\\
439 6.98191107785943e-09\\
601 2.26207943922867e-13\\
824 3.24490438592728e-13\\
1129 8.45712389976179e-14\\
1547 1.34572911416789e-13\\
2120 1.03667107529034e-13\\
2905 2.77181058109711e-13\\
3981 1.17253429631273e-13\\
};
\addplot [
color=green!50!black,
dashed,
mark=triangle,
mark options={solid},
forget plot
]
table[row sep=crcr]{
10 8.73483450106512\\
14 12.8636722779264\\
19 23.8848511479506\\
26 9.78007842950505\\
35 5.05714549706016\\
48 2.52899543504233\\
66 9.66176220165118\\
91 4.36672927661848\\
124 4.14177613790697\\
170 3.54058396243311\\
234 5.58058610897885\\
320 5.78235721199833\\
439 2.3059753893085\\
601 4.25984698679342\\
824 1.2435814778538e-08\\
1129 1.23207111666045e-13\\
1547 2.85826919970679e-13\\
2120 2.30649076122025e-13\\
2905 4.49645084805747e-14\\
3981 1.17100773760721e-13\\
};
\addplot [
color=green!50!black,
dashed
]
table[row sep=crcr]{
10 8.73483450106512\\
14 12.8636722779264\\
19 23.8848511479506\\
26 9.78007842950505\\
35 5.05714549706016\\
48 2.52899543504233\\
66 9.66176220165118\\
91 4.36672927661848\\
124 4.14177613790697\\
170 3.54058396243311\\
234 5.58058610897885\\
320 5.78235721199833\\
439 2.3059753893085\\
601 4.25984698679342\\
824 1.2435814778538e-08\\
1129 1.23207111666045e-13\\
1547 2.85826919970679e-13\\
2120 2.30649076122025e-13\\
2905 4.49645084805747e-14\\
3981 1.17100773760721e-13\\
};
\addlegendentry{\TRID};

\end{axis}
\end{tikzpicture}%

%% file: img/PAvsTRIDf3res.tex
%
\begin{tikzpicture}

\begin{axis}[%
width=\figurewidth,
height=\figureheight,
scale only axis,
xmode=log,
xmin=10,
xmax=10000,
xminorticks=true,
xlabel={N},
ymode=log,
ymin=1e-12,
ymax=100,
yminorticks=true,
ylabel={\resnorm},
axis x line*=bottom,
axis y line*=left,
legend style={draw=black,fill=white,legend cell align=left}
]
\addplot [
color=blue,
solid,
mark=o,
mark options={solid}
]
table[row sep=crcr]{
10 0.537926350089028\\
14 0.223902816229885\\
19 0.0712760346641269\\
26 0.01513129901988\\
35 0.00217181507114933\\
48 0.000118069882002118\\
66 2.49440078195283e-06\\
91 1.38997897678373e-08\\
124 1.30959854182309e-10\\
170 1.59228397998421e-11\\
234 3.54087593747808e-12\\
320 8.31440195744633e-12\\
439 3.07119065106656e-12\\
601 8.29658652640175e-12\\
824 1.28630281364137e-11\\
1129 1.28598104650666e-11\\
1547 9.97416967386005e-12\\
2120 1.0174749069887e-11\\
2905 3.47097727970203e-11\\
3981 2.49360582068467e-11\\
};
\addlegendentry{T=1.1};

\addplot [
color=blue,
solid,
mark=square,
mark options={solid}
]
table[row sep=crcr]{
10 5.52283292566179\\
14 0.0286230960979303\\
19 0.00515580109469619\\
26 0.00130353715256538\\
35 0.000194517848329052\\
48 2.97463257941987e-05\\
66 2.68515007837142e-06\\
91 2.5036977322456e-07\\
124 1.12539393712343e-08\\
170 8.52391495999432e-10\\
234 9.07161312431661e-11\\
320 2.3935730479574e-11\\
439 2.01582234204795e-11\\
601 2.60963572674651e-11\\
824 2.95429664381678e-11\\
1129 2.4964169732576e-11\\
1547 3.68550269321552e-11\\
2120 3.50584666148501e-11\\
2905 3.68783570402439e-11\\
3981 6.20993429066939e-11\\
};
\addlegendentry{T=2};

\addplot [
color=blue,
solid,
mark=triangle,
mark options={solid}
]
table[row sep=crcr]{
10 6.44705884271488\\
14 0.113157021615259\\
19 0.0498599110532642\\
26 0.0167424330993279\\
35 0.00507793390995338\\
48 0.00116377693323929\\
66 0.000186869071918132\\
91 2.71712834305145e-05\\
124 3.74612879986318e-06\\
170 1.97802481450449e-07\\
234 1.65509033709335e-08\\
320 1.08705438835279e-09\\
439 1.70673934239586e-10\\
601 2.97301443744406e-11\\
824 2.80779020912906e-11\\
1129 3.68338995600398e-11\\
1547 2.73053285209489e-11\\
2120 2.67502778001131e-11\\
2905 5.52307887304397e-11\\
3981 3.62236825727418e-11\\
};
\addlegendentry{T=3.8};

\addplot [
color=blue,
solid
]
table[row sep=crcr]{
10 6.44705884271488\\
14 0.113157021615259\\
19 0.0498599110532642\\
26 0.0167424330993279\\
35 0.00507793390995338\\
48 0.00116377693323929\\
66 0.000186869071918132\\
91 2.71712834305145e-05\\
124 3.74612879986318e-06\\
170 1.97802481450449e-07\\
234 1.65509033709335e-08\\
320 1.08705438835279e-09\\
439 1.70673934239586e-10\\
601 2.97301443744406e-11\\
824 2.80779020912906e-11\\
1129 3.68338995600398e-11\\
1547 2.73053285209489e-11\\
2120 2.67502778001131e-11\\
2905 5.52307887304397e-11\\
3981 3.62236825727418e-11\\
};
\addlegendentry{\PA};

\addplot [
color=green!50!black,
dashed,
mark=o,
mark options={solid},
forget plot
]
table[row sep=crcr]{
10 0.53026447407993\\
14 0.223417250798083\\
19 0.0704007097701211\\
26 0.0151223469022703\\
35 0.00217175753689602\\
48 0.000118092998938835\\
66 2.4952180552873e-06\\
91 1.38998856869466e-08\\
124 1.30784277786389e-10\\
170 1.46656207991001e-11\\
234 1.39977078250855e-12\\
320 4.0110385419507e-12\\
439 3.80501303732816e-12\\
601 2.245325377476e-12\\
824 5.84429487958513e-12\\
1129 5.81604473557026e-12\\
1547 7.95630774199399e-12\\
2120 1.58362531480832e-11\\
2905 1.11597292021485e-11\\
3981 3.25322033141632e-11\\
};
\addplot [
color=green!50!black,
dashed,
mark=square,
mark options={solid},
forget plot
]
table[row sep=crcr]{
10 0.121067242734733\\
14 0.0286233141708507\\
19 0.00702139591767525\\
26 0.0012984270725723\\
35 0.000191271920021892\\
48 2.94503649900256e-05\\
66 2.67883488377255e-06\\
91 1.27912233718188e-07\\
124 1.05326520839012e-08\\
170 5.94020832603354e-10\\
234 5.34434720199203e-11\\
320 1.2835954555533e-11\\
439 3.44968683575793e-12\\
601 3.70015129649036e-12\\
824 5.06794700326315e-12\\
1129 3.29869546020276e-12\\
1547 6.80522343365575e-12\\
2120 1.05338018064236e-12\\
2905 5.01643466753412e-12\\
3981 3.19211555738807e-12\\
};
\addplot [
color=green!50!black,
dashed,
mark=triangle,
mark options={solid},
forget plot
]
table[row sep=crcr]{
10 0.287338600782316\\
14 0.135073285026302\\
19 0.0498704681342631\\
26 0.0207190270341209\\
35 0.0049075232176628\\
48 0.00115526442874447\\
66 0.000182602925213395\\
91 2.6547210740844e-05\\
124 2.16499754661357e-06\\
170 1.61722546465996e-07\\
234 7.68077690477527e-09\\
320 7.48952899699308e-10\\
439 3.0718538853007e-11\\
601 1.64934751015163e-11\\
824 3.5687054653199e-12\\
1129 2.46203900724532e-12\\
1547 6.79279161457602e-12\\
2120 2.26946206105211e-12\\
2905 8.54079646070921e-12\\
3981 3.852972753947e-12\\
};
\addplot [
color=green!50!black,
dashed
]
table[row sep=crcr]{
10 0.287338600782316\\
14 0.135073285026302\\
19 0.0498704681342631\\
26 0.0207190270341209\\
35 0.0049075232176628\\
48 0.00115526442874447\\
66 0.000182602925213395\\
91 2.6547210740844e-05\\
124 2.16499754661357e-06\\
170 1.61722546465996e-07\\
234 7.68077690477527e-09\\
320 7.48952899699308e-10\\
439 3.0718538853007e-11\\
601 1.64934751015163e-11\\
824 3.5687054653199e-12\\
1129 2.46203900724532e-12\\
1547 6.79279161457602e-12\\
2120 2.26946206105211e-12\\
2905 8.54079646070921e-12\\
3981 3.852972753947e-12\\
};
\addlegendentry{\TRID};

\end{axis}
\end{tikzpicture}%

%% file: img/PAvsTRIDf4res.tex
%
\begin{tikzpicture}

\begin{axis}[%
width=\figurewidth,
height=\figureheight,
scale only axis,
xmode=log,
xmin=10,
xmax=10000,
xminorticks=true,
xlabel={N},
ymode=log,
ymin=1e-05,
ymax=10,
yminorticks=true,
ylabel={\resnorm},
axis x line*=bottom,
axis y line*=left,
legend style={draw=black,fill=white,legend cell align=left}
]
\addplot [
color=blue,
solid,
mark=o,
mark options={solid}
]
table[row sep=crcr]{
10 0.0197470650708067\\
14 0.0143428288223777\\
19 0.0106962135930397\\
26 0.00787481615424329\\
35 0.00586055145322476\\
48 0.00429647596638067\\
66 0.00313616304640601\\
91 0.00227983995649175\\
124 0.00167824574242908\\
170 0.00122917730918889\\
234 0.000894933211999435\\
320 0.000656245367664946\\
439 0.000479065371678134\\
601 0.000350409397207894\\
824 0.000255799930997966\\
1129 0.000186835676141707\\
1547 0.000136419729542238\\
2120 9.95853974297705e-05\\
2905 7.2696159376295e-05\\
3981 5.30565386778334e-05\\
};
\addlegendentry{T=1.1};

\addplot [
color=blue,
solid,
mark=square,
mark options={solid}
]
table[row sep=crcr]{
10 0.377967764878688\\
14 0.0260120567639101\\
19 0.0873421669473711\\
26 0.0382032504519318\\
35 0.040965453812767\\
48 0.0179697434406193\\
66 0.0120563828828961\\
91 0.00734646916817689\\
124 0.00425416540991189\\
170 0.00247965536124108\\
234 0.00156538077564455\\
320 0.00115869660901575\\
439 0.000853286825614024\\
601 0.000627497972364188\\
824 0.000460108731380228\\
1129 0.00033708705981806\\
1547 0.000246628289392743\\
2120 0.000180301499969088\\
2905 0.000131771397843863\\
3981 9.62531776525416e-05\\
};
\addlegendentry{T=2};

\addplot [
color=blue,
solid,
mark=triangle,
mark options={solid}
]
table[row sep=crcr]{
10 0.426879428781529\\
14 2.27264660001063\\
19 0.865143594749764\\
26 0.627657340882538\\
35 0.143244052396592\\
48 0.0791941290758613\\
66 0.0592053974212004\\
91 0.0213497371718372\\
124 0.0104745782367025\\
170 0.0064283810463051\\
234 0.0032066720896278\\
320 0.00259604893341151\\
439 0.00163709558691937\\
601 0.00116513183213321\\
824 0.000860005962725309\\
1129 0.00063296723531379\\
1547 0.000464533189793157\\
2120 0.000340455672536963\\
2905 0.000249229650713132\\
3981 0.000182276355361832\\
};
\addlegendentry{T=3.8};

\addplot [
color=blue,
solid
]
table[row sep=crcr]{
10 0.426879428781529\\
14 2.27264660001063\\
19 0.865143594749764\\
26 0.627657340882538\\
35 0.143244052396592\\
48 0.0791941290758613\\
66 0.0592053974212004\\
91 0.0213497371718372\\
124 0.0104745782367025\\
170 0.0064283810463051\\
234 0.0032066720896278\\
320 0.00259604893341151\\
439 0.00163709558691937\\
601 0.00116513183213321\\
824 0.000860005962725309\\
1129 0.00063296723531379\\
1547 0.000464533189793157\\
2120 0.000340455672536963\\
2905 0.000249229650713132\\
3981 0.000182276355361832\\
};
\addlegendentry{\PA};

\addplot [
color=green!50!black,
dashed,
mark=o,
mark options={solid},
forget plot
]
table[row sep=crcr]{
10 0.0197934174119321\\
14 0.0143297621168931\\
19 0.0106759661457762\\
26 0.00786302550882495\\
35 0.00586141414018967\\
48 0.0043004371258022\\
66 0.00313670231399513\\
91 0.00227986060816864\\
124 0.00167825531751005\\
170 0.00122917296474423\\
234 0.000894949765581114\\
320 0.00065615922852576\\
439 0.000479064019743718\\
601 0.000350347413448339\\
824 0.000255800304642138\\
1129 0.000186816557942157\\
1547 0.000136410654088587\\
2120 9.9581136855654e-05\\
2905 7.269379865907e-05\\
3981 5.30543124078653e-05\\
};
\addplot [
color=green!50!black,
dashed,
mark=square,
mark options={solid},
forget plot
]
table[row sep=crcr]{
10 0.0232200715395194\\
14 0.0259833851189904\\
19 0.033664556940751\\
26 0.0372571010864847\\
35 0.0392800774187735\\
48 0.0168874967382988\\
66 0.0104349400277887\\
91 0.0084469050782295\\
124 0.00388750746974605\\
170 0.0023356960305153\\
234 0.00156444044273221\\
320 0.00115861296653816\\
439 0.000851986600967171\\
601 0.000626831987874063\\
824 0.000459786271644322\\
1129 0.000336761352288694\\
1547 0.000246467924028703\\
2120 0.000180249160929474\\
2905 0.000131725001576804\\
3981 9.62287996366362e-05\\
};
\addplot [
color=green!50!black,
dashed,
mark=triangle,
mark options={solid},
forget plot
]
table[row sep=crcr]{
10 0.541937456065889\\
14 0.382863386467619\\
19 0.837104896598094\\
26 0.182565305578447\\
35 0.199788092894297\\
48 0.0744495551018857\\
66 0.0599636975584873\\
91 0.0213143749721116\\
124 0.0156449487057112\\
170 0.00871420220597737\\
234 0.00559585586085717\\
320 0.00237847943023523\\
439 0.00191969710275161\\
601 0.00116488073817174\\
824 0.000858565365158672\\
1129 0.000632106934995223\\
1547 0.000464133117312333\\
2120 0.000340163373912245\\
2905 0.000249020207514575\\
3981 0.000182173964777095\\
};
\addplot [
color=green!50!black,
dashed
]
table[row sep=crcr]{
10 0.541937456065889\\
14 0.382863386467619\\
19 0.837104896598094\\
26 0.182565305578447\\
35 0.199788092894297\\
48 0.0744495551018857\\
66 0.0599636975584873\\
91 0.0213143749721116\\
124 0.0156449487057112\\
170 0.00871420220597737\\
234 0.00559585586085717\\
320 0.00237847943023523\\
439 0.00191969710275161\\
601 0.00116488073817174\\
824 0.000858565365158672\\
1129 0.000632106934995223\\
1547 0.000464133117312333\\
2120 0.000340163373912245\\
2905 0.000249020207514575\\
3981 0.000182173964777095\\
};
\addlegendentry{\TRID};

\end{axis}
\end{tikzpicture}%

%% file: img/robustnessaccuracyres.tex
%
\begin{tikzpicture}

\begin{axis}[%
width=\figurewidth,
height=\figureheight,
scale only axis,
xmode=log,
xmin=10,
xmax=100000,
xminorticks=true,
xlabel={N},
ymode=log,
ymin=1e-15,
ymax=1,
yminorticks=true,
ylabel={\resnorm},
axis x line*=bottom,
axis y line*=left,
legend style={draw=black,fill=white,legend cell align=left}
]
\addplot [
color=blue,
solid,
mark=o,
mark options={solid}
]
table[row sep=crcr]{
10 0.0531807358765876\\
16 0.00615903774264162\\
26 0.000408936993750461\\
43 3.91862243626299e-06\\
70 2.77125671265108e-09\\
113 6.28646257640212e-14\\
183 1.2935326550481e-13\\
298 4.9581894891763e-13\\
483 4.46995796966386e-13\\
785 6.6884549450764e-13\\
1274 4.94134166761627e-13\\
2069 7.59556596728928e-13\\
3360 1.573009806688e-12\\
5456 1.04149781814746e-12\\
8859 2.3448578201259e-12\\
14384 2.8362326221869e-12\\
23357 5.19041899139554e-12\\
37927 7.39195919850547e-12\\
61585 4.21775319917343e-12\\
100000 7.29727921018788e-12\\
};
\addlegendentry{T=1.1};

\addplot [
color=blue,
solid,
mark=square,
mark options={solid}
]
table[row sep=crcr]{
10 0.394587476907683\\
16 9.77146144827401e-12\\
26 1.07963706387514e-13\\
43 1.86134350634786e-13\\
70 2.61404733795545e-13\\
113 7.44512553915653e-13\\
183 3.01670361614247e-13\\
298 5.49760392670968e-13\\
483 1.32617912812656e-12\\
785 2.25734352396141e-12\\
1274 8.36624479880461e-13\\
2069 2.55679851986032e-12\\
3360 3.77192919042289e-12\\
5456 1.61631820942905e-12\\
8859 2.28640606119951e-12\\
14384 4.62720215938449e-12\\
23357 8.08728933560562e-12\\
37927 4.73134616211603e-12\\
61585 5.83174477941353e-12\\
100000 9.84832470836332e-12\\
};
\addlegendentry{T=2};

\addplot [
color=blue,
solid,
mark=triangle,
mark options={solid}
]
table[row sep=crcr]{
10 0.890883464067299\\
16 5.0220290571427e-13\\
26 2.60682820457813e-13\\
43 1.65495010465939e-13\\
70 5.06635548384291e-13\\
113 4.45400744741175e-13\\
183 8.77537893222681e-13\\
298 6.40307576299907e-13\\
483 1.11795782898876e-12\\
785 1.31017606537133e-12\\
1274 1.66867983752734e-12\\
2069 1.55495930722943e-12\\
3360 6.06050382964229e-12\\
5456 1.39301388526079e-12\\
8859 5.44755032341975e-12\\
14384 1.09977581866191e-11\\
23357 3.80374718742173e-12\\
37927 3.33481732433297e-12\\
61585 6.45357294558205e-12\\
100000 1.19352741638465e-11\\
};
\addlegendentry{T=3.8};

\addplot [
color=blue,
solid
]
table[row sep=crcr]{
10 0.890883464067299\\
16 5.0220290571427e-13\\
26 2.60682820457813e-13\\
43 1.65495010465939e-13\\
70 5.06635548384291e-13\\
113 4.45400744741175e-13\\
183 8.77537893222681e-13\\
298 6.40307576299907e-13\\
483 1.11795782898876e-12\\
785 1.31017606537133e-12\\
1274 1.66867983752734e-12\\
2069 1.55495930722943e-12\\
3360 6.06050382964229e-12\\
5456 1.39301388526079e-12\\
8859 5.44755032341975e-12\\
14384 1.09977581866191e-11\\
23357 3.80374718742173e-12\\
37927 3.33481732433297e-12\\
61585 6.45357294558205e-12\\
100000 1.19352741638465e-11\\
};
\addlegendentry{\PA};

\addplot [
color=green!50!black,
dashed,
mark=o,
mark options={solid},
forget plot
]
table[row sep=crcr]{
10 0.0531571953324746\\
16 0.00609586540176621\\
26 0.00040915451792517\\
43 3.91856062476048e-06\\
70 2.77255352010287e-09\\
113 5.11812952922868e-14\\
183 7.87227529165893e-14\\
298 7.06103622149711e-14\\
483 3.23186230708111e-13\\
785 1.97730750135583e-13\\
1274 6.45372651180028e-13\\
2069 8.47878696711571e-13\\
3360 1.74772224789333e-12\\
5456 1.54710588444614e-12\\
8859 3.37731318144009e-13\\
14384 5.77661847087874e-12\\
23357 5.08228428177776e-12\\
37927 3.41087279768158e-11\\
61585 1.44500032192032e-11\\
100000 1.28662482700631e-11\\
};
\addplot [
color=green!50!black,
dashed,
mark=square,
mark options={solid},
forget plot
]
table[row sep=crcr]{
10 3.16895077179247e-07\\
16 9.77223857413704e-12\\
26 2.50128616195066e-15\\
43 3.19748509341321e-14\\
70 6.16173783685267e-14\\
113 1.46438439108877e-13\\
183 3.99125203351682e-13\\
298 2.78667878587246e-14\\
483 7.07212219343456e-14\\
785 1.30784315884194e-13\\
1274 1.98618936086366e-13\\
2069 6.37273780017267e-14\\
3360 2.88325021850797e-13\\
5456 5.03042053374987e-13\\
8859 5.23026066956652e-13\\
14384 2.11886065216227e-12\\
23357 6.99552871173874e-13\\
37927 2.62545563721336e-12\\
61585 2.45825612569042e-12\\
100000 3.70936632293451e-12\\
};
\addplot [
color=green!50!black,
dashed,
mark=triangle,
mark options={solid},
forget plot
]
table[row sep=crcr]{
10 0.000123881310877428\\
16 4.8594462384163e-13\\
26 2.15605336981721e-13\\
43 2.55512808075999e-15\\
70 4.20330528600909e-13\\
113 1.42552636547861e-13\\
183 2.15161459043172e-13\\
298 1.88627271478793e-13\\
483 2.19380231004398e-13\\
785 9.62598430494788e-14\\
1274 3.25517420988359e-13\\
2069 6.4727717653868e-14\\
3360 1.08357791478642e-13\\
5456 3.07977983780254e-13\\
8859 2.51473920345437e-13\\
14384 7.11347605711325e-13\\
23357 4.97519899462676e-13\\
37927 4.46565023881245e-12\\
61585 2.55910141827839e-12\\
100000 1.06061008625841e-11\\
};
\addplot [
color=green!50!black,
dashed
]
table[row sep=crcr]{
10 0.000123881310877428\\
16 4.8594462384163e-13\\
26 2.15605336981721e-13\\
43 2.55512808075999e-15\\
70 4.20330528600909e-13\\
113 1.42552636547861e-13\\
183 2.15161459043172e-13\\
298 1.88627271478793e-13\\
483 2.19380231004398e-13\\
785 9.62598430494788e-14\\
1274 3.25517420988359e-13\\
2069 6.4727717653868e-14\\
3360 1.08357791478642e-13\\
5456 3.07977983780254e-13\\
8859 2.51473920345437e-13\\
14384 7.11347605711325e-13\\
23357 4.97519899462676e-13\\
37927 4.46565023881245e-12\\
61585 2.55910141827839e-12\\
100000 1.06061008625841e-11\\
};
\addlegendentry{\TRID};

\end{axis}
\end{tikzpicture}%

%% file: img/robustnessaccuracy2res.tex
%
\begin{tikzpicture}

\begin{axis}[%
width=\figurewidth,
height=\figureheight,
scale only axis,
xmode=log,
xmin=10,
xmax=100000,
xminorticks=true,
xlabel={N},
ymode=log,
ymin=1e-14,
ymax=100,
yminorticks=true,
ylabel={\resnorm},
axis x line*=bottom,
axis y line*=left,
legend style={draw=black,fill=white,legend cell align=left}
]
\addplot [
color=blue,
solid,
mark=o,
mark options={solid}
]
table[row sep=crcr]{
10 0.0530386552097869\\
16 0.00685408849179682\\
26 0.000195472192117954\\
43 1.2010836245728e-06\\
70 9.56151224862387e-09\\
113 4.05484369012485e-13\\
183 1.2847614738395e-13\\
298 1.96862984801438e-13\\
483 4.69721007924619e-13\\
785 4.57633623441084e-13\\
1274 5.1875981920638e-13\\
2069 2.89954046640475e-12\\
3360 3.22413724628883e-12\\
5456 1.02048935806763e-12\\
8859 7.62088177331593e-12\\
14384 8.80208046956841e-12\\
23357 9.69428568415752e-12\\
37927 3.18603562738825e-11\\
61585 3.01272730449565e-11\\
100000 6.81459919807832e-11\\
};
\addlegendentry{T=1.1};

\addplot [
color=blue,
solid,
mark=square,
mark options={solid}
]
table[row sep=crcr]{
10 0.394587476907683\\
16 1.32759248041185e-10\\
26 3.01974152665933e-13\\
43 4.35143220241958e-13\\
70 4.04013459897274e-13\\
113 2.21446032330708e-13\\
183 7.92045484855036e-13\\
298 6.12345581189677e-13\\
483 2.91113132562722e-12\\
785 6.413298856007e-12\\
1274 7.60986486251952e-12\\
2069 3.94192749762351e-12\\
3360 6.9653927507314e-12\\
5456 3.56306584195921e-12\\
8859 8.74613314737076e-12\\
14384 1.69173489907366e-11\\
23357 7.16427281880956e-12\\
37927 1.22587859583703e-11\\
61585 9.10470047092167e-12\\
100000 1.50567739168403e-11\\
};
\addlegendentry{T=2};

\addplot [
color=blue,
solid,
mark=triangle,
mark options={solid}
]
table[row sep=crcr]{
10 0.890883464067299\\
16 8.06528199572999e-05\\
26 0.00108391652831052\\
43 0.0314829774708582\\
70 2.00663861273711\\
113 52.2496778119192\\
183 26.3713342445533\\
298 15.988864663458\\
483 7.46253670039337\\
785 7.5510112879932\\
1274 20.6237453707205\\
2069 15.502439373191\\
3360 15.1711466719705\\
5456 13.5516344915609\\
8859 13.1831973836785\\
14384 14.6612579776143\\
23357 14.635494959559\\
37927 6.15292430753356\\
61585 11.4869261123443\\
100000 12.1561934228677\\
};
\addlegendentry{T=3.8};

\addplot [
color=blue,
solid
]
table[row sep=crcr]{
10 0.890883464067299\\
16 8.06528199572999e-05\\
26 0.00108391652831052\\
43 0.0314829774708582\\
70 2.00663861273711\\
113 52.2496778119192\\
183 26.3713342445533\\
298 15.988864663458\\
483 7.46253670039337\\
785 7.5510112879932\\
1274 20.6237453707205\\
2069 15.502439373191\\
3360 15.1711466719705\\
5456 13.5516344915609\\
8859 13.1831973836785\\
14384 14.6612579776143\\
23357 14.635494959559\\
37927 6.15292430753356\\
61585 11.4869261123443\\
100000 12.1561934228677\\
};
\addlegendentry{\PA};

\addplot [
color=green!50!black,
dashed,
mark=o,
mark options={solid},
forget plot
]
table[row sep=crcr]{
10 0.0531571953324746\\
16 0.00682344887121991\\
26 0.000194905969427728\\
43 1.20108923917783e-06\\
70 9.56152024222945e-09\\
113 4.0750767887964e-13\\
183 1.39778320158436e-13\\
298 2.30982028902967e-13\\
483 5.09814511416913e-13\\
785 1.76747515326387e-13\\
1274 6.34061793374897e-13\\
2069 2.56872410008618e-12\\
3360 5.25324746324052e-12\\
5456 4.46598759153368e-12\\
8859 4.47108254310346e-12\\
14384 1.98583526767655e-11\\
23357 2.29491317716991e-11\\
37927 3.82182063975498e-11\\
61585 3.98686167035112e-11\\
100000 2.05746634960754e-10\\
};
\addplot [
color=green!50!black,
dashed,
mark=square,
mark options={solid},
forget plot
]
table[row sep=crcr]{
10 3.1689507734578e-07\\
16 1.32758498639196e-10\\
26 2.04725390033701e-13\\
43 1.63208484036017e-14\\
70 1.21347741916985e-13\\
113 2.27605023535536e-14\\
183 9.23705570695775e-14\\
298 5.98410814183373e-14\\
483 7.9381209851236e-14\\
785 3.0464520084054e-13\\
1274 1.89182003691724e-13\\
2069 4.09006235938557e-13\\
3360 3.97317596216789e-13\\
5456 9.02167325112403e-13\\
8859 3.11478638718249e-12\\
14384 2.08999484653643e-12\\
23357 2.97170621447538e-12\\
37927 7.90434389993294e-12\\
61585 7.27424025666444e-12\\
100000 1.17917620112781e-11\\
};
\addplot [
color=green!50!black,
dashed,
mark=triangle,
mark options={solid},
forget plot
]
table[row sep=crcr]{
10 0.000123881320919172\\
16 8.04204096807355e-05\\
26 0.000983243658530652\\
43 0.0301628100224231\\
70 2.1246536071925\\
113 50.0054269537877\\
183 24.9308008210204\\
298 33.5730427618059\\
483 29.8154074024589\\
785 29.9980623829448\\
1274 7.23951403063408\\
2069 3.52012216814226\\
3360 20.3791822382663\\
5456 27.7926814984943\\
8859 32.4335365748197\\
14384 1.25882289424092\\
23357 17.857886796434\\
37927 3.59632167568374\\
61585 28.1028417716474\\
100000 28.9630815167909\\
};
\addplot [
color=green!50!black,
dashed
]
table[row sep=crcr]{
10 0.000123881320919172\\
16 8.04204096807355e-05\\
26 0.000983243658530652\\
43 0.0301628100224231\\
70 2.1246536071925\\
113 50.0054269537877\\
183 24.9308008210204\\
298 33.5730427618059\\
483 29.8154074024589\\
785 29.9980623829448\\
1274 7.23951403063408\\
2069 3.52012216814226\\
3360 20.3791822382663\\
5456 27.7926814984943\\
8859 32.4335365748197\\
14384 1.25882289424092\\
23357 17.857886796434\\
37927 3.59632167568374\\
61585 28.1028417716474\\
100000 28.9630815167909\\
};
\addlegendentry{\TRID};

\end{axis}
\end{tikzpicture}%